\documentclass{llncs}
\usepackage{latexsym, textcomp}
\usepackage{amssymb, amsmath, mathrsfs,amsfonts}
\usepackage{graphicx}
\usepackage{epstopdf}

\newtheorem{fact}[theorem]{Fact}

\newcommand{\A}{\mathcal A}
\newcommand{\B}{\mathcal B}
\newcommand{\C}{\mathcal C}
\newcommand{\F}{\mathcal F}
\newcommand{\G}{\mathcal G}
\renewcommand{\L}{\mathcal L}
\newcommand{\M}{\mathcal M}
\renewcommand{\P}{\mathcal P}
\renewcommand{\S}{\mathcal S}
\newcommand{\T}{\mathcal T}
\newcommand{\N}{\mathbb N}
\newcommand{\Q}{\mathbb Q}
\newcommand{\Z}{\mathbb Z}
\newcommand{\lqt}{\textquotedblleft}
\newcommand{\rqt}{\textquotedblright~}
\newcommand{\la}{\langle}
\newcommand{\ra}{\rangle}

\begin{document}

\title{Three Lectures on Automatic Structures}

\author{Bakhadyr Khoussainov\inst{1} \and Mia Minnes \inst{2}}
\authorrunning{Khoussainov and Minnes} 
\institute{Department of Computer Science\\ University of Auckland\\Auckland, New Zealand\\ \email{bmk@cs.auckland.ac.nz}
\and
Mathematics Department\\ Cornell University\\Ithaca, New York 14853\\ \email{minnes@math.cornell.edu}}

\maketitle 

{\bf \large Preface}

\bigskip

This paper grew out of three tutorial lectures on automatic structures 
given at the Logic Colloquium 2007 in Wroc\l aw (Poland).
The paper will follow the outline of the tutorial lectures, supplementing some material along the way.  We discuss variants of automatic structures related to several models 
of computation: word automata, tree automata, B\"uchi automata, and Rabin automata.
Word automata process finite strings, tree automata process finite labeled trees,
B\"uchi automata process infinite strings, and Rabin automata process infinite binary
labeled trees. Finite automata are the most commonly known in the general
computer science community. Such an automaton reads finite input strings from
left to right, making state transitions along the way. Depending on its last state after processing a given string, the automaton either accepts or
rejects the input string. Automatic structures are mathematical objects which can be represented by (word, tree, B\"uchi, or Rabin) automata.
The study of properties of automatic structures is a relatively new and very active area of research. 

\smallskip

We begin with some motivation and history for studying automatic structures. 
We introduce definitions of  automatic structures,  present examples, and 
discuss decidability and definability theorems.  
Next, we concentrate on finding natural isomorphism
invariants
for classes of automatic structures. These classes include well-founded
partial
orders, Boolean algebras, linear orders, trees, and finitely generated
groups. 
Finally, we address the issue of complexity for automatic structures. In order to
measure
complexity of automatic structures we involve topology (via the Borel
hierarchy),
model theory (Scott ranks), computability theory ($\Sigma_1^1$-completeness), and computational complexity (the class P).

\smallskip

This paper consists of three sections based on the tutorial lectures.
The first lecture provides motivating  questions and historical context,
formal definitions of the different types of automata involved, examples of automatic structures, and decidability and definability  results about automatic structures.
The second lecture considers  tree and Rabin automatic structures, and outlines techniques 
for proving non-automaticity.  We then turn our attention to 
the study of algorithmic and structural properties of automatic trees,  Boolean algebras, 
and finitely generated groups.  The final lecture  presents a framework for reducing certain questions about computable structures to questions about automatic structures.  These reductions have been used to show that, in some cases, the sharp bound on complexity of automatic structures is as high as the bounds for computable structures.  We conclude by looking at Borel structures from descriptive set theory and connecting them to B\"uchi and Rabin automatic structures.

\section{
{\em Basics}}\label{sec:Background}

\subsection{Motivating questions}\label{subsec:Questions}

The study of structures has played a central role in the development of logic.  
In the course of this study, several themes have been pursued. We will see how questions related to each of these themes is addressed in the particular case of automatic structures.

\smallskip

{\it The isomorphism problem}. One of the major tasks in the study of
structures is concerned with the classification of isomorphism types.  The isomorphism problem asks: \lqt given two structures, are they isomorphic?\rqt
Ideally, we would like to define invariants
which are simpler than their associated structures and yet describe the
structures up to isomorphism. An example of such an invariant is the dimension 
of vector spaces.  For algebraically closed fields, such invariants are the 
characteristics of the fields along with their transendence degrees.  Despite this, there are well-studied classes of structures (such as abelian groups, Boolean algebras, linearly ordered sets, algebras, and graphs) for which it is
impossible to give simple isomorphism invariants.  Thus, in general, 
there is no solution to the isomorphism problem.

\smallskip

{\it The elementary equivalence problem}. Since there is no general solution to the isomorphism problem, we 
may approximate it in various ways.  One natural approximation comes from logic 
in the form of  elementary equivalence.  
While elementary equivalence may be defined with respect to any logic, we refer to 
elementary equivalence with respect to either the first-order logic or the monadic second-order (MSO) logic. 
There are several good examples of positive results in this setting: elementarily equivalent (in the first-order logic)  Boolean algebras can be characterized via Ershov-Tarski
invariants \cite{Ers64}, \cite{Tar49}; 
elementarily equivalent abelian groups (in the first-order logic) can be characterized via Shmelev invariants \cite{Shmelev}.   Moreover, there is a body of work devoted to understanding when elementary equivalence (or any of its weaker versions) imply isomorphism. An example here is Pop's conjecture: if two finitely generated 
fields are elementarily equivalent then they are isomorphic. See \cite{scanlon} for a solution of this problem.

\smallskip

{\it The model checking problem}.  The model checking problem is the name given by 
the computer science community to a particular instance of the elementary equivalence
problem.  The problem asks, for a given sentence $\varphi$ and a given structure, whether the structure satisfies $\varphi$.  For example, the sentence $\varphi$
can be the axiomatization of groups or can state that a graph is connected (in which case a stronger logic is used
than the first-order logic). 

\smallskip

{\it Deciding the theories of structures}. This is a traditional topic in logic
that seeks algorithms to decide the full first-order theory (or MSO theory)  of a given structure. Clearly, the problem can be thought of as a uniform version of the model checking problem. Often, to prove that a structure has a decidable first-order theory, 
one translates formulas into an equivalent form that has a small number of alternations of
quantifiers (or no quantifiers at all), and then shows that the resulting simpler formulas are easy to decide. This approach is intimately related to the next theme in our list.

\smallskip

{\it Characterization of definable relations}. Here, we would like to
understand properties of definable relations in a given structure.
A classical example is real closed fields, where
a relation is definable in the first-order logic if and only if it is a Boolean
combination of polynomial equations and inequations \cite{Tar48}.
Usually, but not always, a good understanding of definable relations yields a quantifier 
elimination procedure and shows that the theory of the structure is decidable.

\smallskip

It is apparent that all of the above problems are interrelated.  We will discuss these problems and their refinements with respect to classes of automatic structures.

\subsection{Background}\label{subsec:history}

The theory of automatic structures can be motivated from the standpoint of 
 computable structures.
The theory of computable structures has a
rich history going back to van der Waerden who informally introduced the
concept of an explicitly given field in the 1930s. This concept was formalized by Fr\"olich
and Shepherdson \cite{FS56}, and later extended by Rabin \cite{Rab60} in the 1950s. In the 1960s, Mal'cev initiated a
systematic study of the theory (see, for example \cite{Mal61}).  Later, computability theoretic techniques  were introduced in order to study the
effective content of algebraic and model theoretic results and
constructions. See \cite{H1}, \cite{H2} for the current state and historical background of the area.

\smallskip

A computable structure is one whose domain and basic relations are all computable by Turing machines.  As a point of comparison with finite automata, we note that 
the Turing machine represents unbounded resources and unbounded read-passes over the data.  The themes outlined in Subsection \ref{subsec:Questions} have been recast
with computability considerations in mind and have given rise to the development  of the theory of computable structures.  For example, one may ask whether a given structure 
is computable.  This is clearly a question about the isomorphism type of the structures
since it asks whether the isomorphism type of the structure contains a computable structure.  So, the study of the computable isomorphism types of structures is a  refinement of the isomorphism problem for structures. 
Another major theme  is devoted to understanding the complexity of  natural problems like the isomorphism problem and the model checking problem. In this case, complexity is often measured in terms of Turing degrees.  We again refer the reader to \cite{H1}, \cite{H2} and the survey paper \cite{Hariz98}.

\smallskip

In the 1980s, as part of their feasible mathematics program, Nerode and Remmel \cite{NeRe90} suggested the study of polynomial-time structures.  A structure is said to be polynomial-time if its domain and relations can be recognized by Turing machines
that run in polynomial time. An important yet simple result (Cenzer and Remmel, \cite{CeRe91}) is
that every computable purely relational structure is computably
isomorphic to a polynomial-time structure. While this result is positive, it
implies that solving questions about the class polynomial-time structures 
is as hard as solving them for the class of computable structures. For instance, the problem of classifying the
isomorphism types of polynomial-time structures is as hard as that of classifying
the isomorphism types of computable structures. 

\smallskip

Since polynomial-time structures and computable structures yielded similar complexity results, greater restrictions on models of computations were imposed.
In 1995, Khoussainov and Nerode suggested bringing in models of computations that have
less computational power than polynomial-time Turing machines.  The hope was that if these weaker machines were used to represent the domain and basic relations, 
then perhaps isomorphism invariants could be more easily understood.   Specifically, they suggested the use of finite state machines (automata) as the basic computation model.
Indeed, the project has been successful as we discuss below.

\smallskip

The idea of using automata to study structures goes back to the
work of B\"uchi. B\"uchi   \cite{Buc60}, \cite{Buc62} used automata to prove the decidability of of a theory called $S1S$ (monadic second-order theory of the natural numbers with one successor).  Rabin \cite{Rab69} then used automata to prove that the monadic second-order theory of two
successor functions, $S2S$, is also decidable. 
In the realm of logic, these results have been used to prove decidability of first-order or MSO theories.   B\"uchi knew that 
 automata and Presburger arithmetic (the first-order theory of the 
natural numbers with addition) are closely connected.  He used automata to give a simple (non quantifier elimination) proof of the decidability of Presburger arithmetic.  
Capturing this notion, Hodgson \cite{H82} defined automata decidable theories in 1982. 
While he coined the definition of automatic structures, throughout the 1980s his work remained unnoticed. In 1995, Khoussainov and Nerode \cite{KhN95} rediscovered 
the concept of automatic structure and initiated the systematic study of the area.

\smallskip

Thurston observed that many finitely generated groups associated with
3-manifolds are finitely presented groups with the property that finite automata recognize equality of
words and the graphs of the operations of left multiplication by a generator;
these are the Thurston  automatic groups.  These groups yield rapid algorithms 
\cite{EpsThurs} for computing topological and algebraic 
properties of interest (such as the word problem). In this development, a group is regarded
as a unary algebra including one unary operation for each generator, the
operation being left multiplication of a word by that generator.  
Among these groups are Coxeter
groups, braid groups, Euclidean groups, and others. The literature is extensive, and we do not
discuss it here. 
We emphasize that Thurston automatic groups differ from automatic groups in our sense; in particular, the vocabulary of the associated structures is starkly different.   Thurston automatic groups are represented as structures whose relations are all unary operations (corresponding to left multiplication).  
On the other hand, an automatic group in our sense deals with the group operation itself (a binary relation) and hence
must satisfy the constraint that the graph of this operation be recognizable
by a finite automaton. The Thurston requirement for automaticity applies
only to finitely generated groups but includes a wider class of finitely
generated groups than what we call automatic groups. Unlike our definition,
Thurston's includes groups whose binary operation of multiplication is not
recognizable by a finite automaton.

\smallskip

In the computer science community, an interest in automatic structures comes from 
problems related to model checking. Model checking is motivated by the quest to prove 
correctness of computer programs. This subject\ allows infinite state automata 
as well as finite state automata. Consult \cite{abl}, \cite{akj}, \cite{bem} for current topics of interest. 
Examples of infinite state automata include concurrency protocols involving
arbitrary number of processes, programs manipulating some infinite sets of
data (such as the integers or reals), pushdown automata, counter automata,
timed automata, Petri-nets, and rewriting systems. Given such an automaton and
a specification (formula) in a formal system, the model checking problem asks
us to compute all the states of the system that satisfy the specification.
Since the state space is infinite, the process of checking the specification
may not terminate. For instance, the standard fixed point computations very often do not
terminate in finite time. Specialized methods are needed to cover even the
problems encountered in practice. Abstraction methods try to represent the
behavior of the system in finite form. Model
checking then reduces to checking a finite representation of states that
satisfy the specification.  Automatic structures arise naturally in infinite state model checking since
both the state space and the transitions of infinite state systems are
usually recognizable by finite automata.   In 2000, Blumensath and Gr\"adel \cite{BluGr00} studied definability problems for automatic structures and the computational complexity of model checking for automatic
structures. 

\smallskip

There has been a series of PhD theses in the area of automatic structures (e.g. 
Rubin \cite{RubPhD}, Blumensath \cite{BluPhD}, B\'ar\'any \cite{BarPhD},
Minnes \cite{MinPhD}, and Liu \cite{LiuPhD}).  A recently published paper of Khoussainov and Nerode \cite{knOQ} discusses open questions in
the study of automatic structures. There are also survey papers on some of the areas in the subject by Nies \cite{NiesSurvey} and  Rubin \cite{RubinSurvey}. 
This avenue of research remains fruitful and active.

\subsection{Automata recognizable languages and relations}\label{sec:AutDefs}

The central models of computation in the development of the theory of automatic
structures are all finite state machines. These include finite automata,
B\"uchi automata, tree automata, and Rabin
automata.  The main distinguishing feature of the different kinds of automata is the kind of input they read.

\begin{definition}
A {\bf B\"uchi automaton} $\M$ is a tuple $(S, \iota, \Delta, F)$, where $S$ is a
finite set of states, $\iota$ is the initial state, $\Delta\subseteq S\times
\Sigma \times S$ (with $\Sigma$ a finite alphabet) is the transition table, and $F\subseteq S$ is
the set of accepting states.
\end{definition}
A B\"uchi automaton can be presented as a finite directed graph with labelled edges. The
vertices of
the graph are the states of the automaton (with designated initial
state and accepting states).  An edge labeled with $\sigma$
connects a state $q$ to a state $q'$ if and only if the triple $(q,\sigma,q')$
is
in $\Delta$. The inputs of a B\"uchi automaton are {\em infinite strings} over the alphabet $\Sigma$.  Let $\alpha=\sigma_0 \sigma_1 \ldots$ be such an infinite
string. The string labels an infinite path through the graph in a natural way.
Such a path is called a run of the automaton on $\alpha$. Formally, a run is
an
infinite sequence $q_0, q_1, \ldots,$ of states such that $q_0$ is the
initial state and $(q_i, \sigma_i, q_{i+1})\in \Delta$ for all $i\in\omega$.
The run is accepting if some accepting state appears in the run
infinitely often. Note that the automaton may have more than one
run 
on a single input $\alpha$. We say that the 
B\"uchi automaton $\M$ accepts a string $\alpha$ if $\M$ has an accepting run on $\alpha$. Thus, acceptance is an existential condition.

\begin{definition}
The set of all infinite words accepted by a B\"uchi automaton $\M$ is called the language of $\M$.  A collection of infinite words is called a {\bf B\"uchi recognizable language} if it is the language of some B\"uchi automaton.
\end{definition}

\begin{example}
The following two languages
over $\Sigma = \{0,1\}$ are B\"uchi recognizable:
\begin{itemize}
\item $\{\alpha \mid \alpha~\text{has
finitely many $1$s}\}$, 
\item $\{\alpha \mid \alpha~\text{has infinitely many $1$s and infinitely
many $0$s}\}$.
\end{itemize}
\end{example}

There are efficient algorithms to decide 
questions about B\"uchi recognizable languages.  A good reference on basic algorithms for B\"uchi automata is Thomas' survey paper \cite{Thomas}.  A central building block in these algorithms is the linear-time decidability of the emptiness problem for B\"uchi automata. The emptiness problem asks for an algorithm that, given a B\"uchi automaton, decides
if the automaton
accepts at least one string.
An equally important result in the development of the
theory of B\"uchi automata says that B\"uchi recognizable languages are closed under union, intersection, and complementation.

\begin{theorem}[B\"uchi; 1960]\label{thm:BucAlgs}\hfill
\begin{enumerate}
\item There is an algorithm that, given a B\"uchi automaton, decides (in
linear time in the size of the automaton) if there is some string accepted by the automaton.
\item The collection of all B\"uchi recognizable languages is closed under the operations of union, intersection, and complementation.
\end{enumerate} 
\end{theorem}
\begin{proof}
The first part of the theorem is easy:
the B\"uchi automaton accepts an infinite
string if and only if there is a path from the initial state to an
accepting state which then loops back to the accepting state. Thus, the emptiness algorithm executes a breadth-first search for such a \lqt lasso\rqt. 
Closure under union and intersection follows from a standard  product construction. It is considerably more difficult to prove closure under complementation. See \cite{Thomas} for a discussion of the issues related to the complementation problem for B\"uchi automata. \qed
\end{proof}

\smallskip

Theorem \ref{thm:BucAlgs} is  true effectively: given two automata we can construct
an automaton that recognizes all the strings accepted by either of the
automata (the union automaton), we can also construct an automaton that
 recognizes all the strings accepted by both automata (the intersection automaton).
 Likewise, there is an algorithm that given an automaton builds a new automaton (called the complement automaton) that accepts all the strings rejected by the original automaton.  We emphasize that complementation constructions have deep significance in
modern automata theory and its applications.  

\smallskip

Since we will be interested in using automata to represent structures, we need to define what it means for a relation to be recognized by an automaton.
Until now, we have discussed B\"uchi automata recognizable {\em sets} of infinite strings.  It is easy to generalize this notion to relations on $\Sigma^\omega$.  Basically, to process a tuple of infinite strings, we read each string in parallel.  More formally, we define the {\bf convolution} of infinite strings $\alpha_1, \ldots, \alpha_k \in \Sigma^\omega$ as the infinite string $c(\alpha_1, \ldots, \alpha_k) \in \left(\Sigma^k\right)^\omega$ whose value at position $i$ is the tuple $\la \alpha_1(i), \ldots, \alpha_k(i) \ra$.  The convolution of a $k$-ary relation $R$, denoted by $c(R)$, is the set of infinite strings over $\Sigma^k$ which are the convolutions of tuples of $R$.  We say that the relation $R$ is B\"uchi recognizable if and only if the set $c(R)$ is B\"uchi recognizable.

\smallskip

We will now describe some operations which preserve automaticity. Let $A$ be
a language of infinite strings over $\Sigma$ and $R$ be 
a relation of arity $k$ on $\Sigma^\omega$.
The {\bf cylindrification} operation on a relation $R$ (with respect to $A$)  
produces the new relation
\[
cyl(R) = \{ \la a_1, \ldots, a_k, a \ra : \la a_1, \ldots, a_k \ra \in R ~\text{and}~a \in A \}.
\]
The {\bf projection}, or $\exists$, operation (with respect to $A$)
is defined by
\[
\exists x_i R = \{ \la a_1, \ldots, a_{i-1}, a_{i+1}, \ldots, a_k \ra  : \exists a \in A ( \la a_1, \ldots, a_{i-1}, a, a_{i+1}, \ldots, a_k \ra \in R ) \}.
\]
The {\bf universal projection}, or $\forall$, operation (with respect to $A$) is defined as
\[
\forall x_i R = \{ \la a_1, \ldots, a_{i-1}, a_{i+1}, \ldots, a_k \ra  : \forall a \in A ( \la a_1, \ldots, a_{i-1}, a, a_{i+1}, \ldots, a_k \ra \in R ) \}.
\]
In all these operations ($cyl, \exists, \forall$), if $R$ and $A$ are both B\"uchi recognizable then the resulting relations are also B\"uchi recognizable. 
The {\bf instantiation} operation is defined by
\[
I(R,c) = \{ \la x_1, \ldots, x_{k-1} \ra : \la x_1, \ldots, x_{k-1},c \ra \in R \}.
\]
If $R$ is B\"uchi automatic then $I(R,c)$ is B\"uchi automatic if and only if $c$ is an ultimately periodic infinite string.  An ultimately periodic word is one of the form $u v^{\omega} = u v v v \cdots$ where $u$ and $v$ are finite strings.  The {\bf rearrangement} operations permute the coordinates of a relation.  If $\pi: k \to k$ is a permutation on the set of $k$ elements and $R$ is a $k$-ary B\"uchi automatic relation, then
\[
\pi R = \{ \la x_{\pi(1)}, \ldots, x_{\pi(k)} \ra : \la x_1, \ldots, x_k \ra \in R \}
\]
is B\"uchi automatic.
The {\bf linkage} operation identifies the last coordinates of some relation with the first coordinates of another.  Given the relations $R$ (of arity $m_1$) and $S$ (of arity $m_2$) and index $i < m_1$, the linkage of $R$ and $S$ on $i$ is the relation of arity $m_2 + i-1$ defined by
\[
L(R^{m_1}, S^{m_2}; i) = \{ \la a_1, \ldots, a_{m_2 + i-1} \ra : \la a_1, \ldots, a_{m_1} \ra \in R ~\&~ \la a_i, \ldots, a_{m_2 + i-1} \ra \in S \}.
\]
For example, $L(R^3, S^4; 2) = \{ \la a_1, a_2, a_3, a_4, a_5 \ra : \la a_1, a_2, a_3 \ra \in R~ \&~ \la a_2, a_3, a_4, a_5 \ra  \in S \}$. If $R$ and $S$ are B\"uchi recognizable relations then so is $L(R, S; i)$.

\smallskip

The closure of B\"uchi recognizable relations under Boolean operations (Theorem \ref{thm:BucAlgs}) connects B\"uchi recognizability and propositional logic.  The $\exists$ and $\forall$ operations bring us to the realm of first-order logic.  We now connect automata with the monadic second-order (MSO) logic of the successor structure $(\omega; S)$. The MSO logic is built on top of the first-order logic as follows. There is one non-logical membership symbol $\in$,
and there are set variables $X,Y, \ldots$ that range over subsets of
$\omega$. Formulas are defined inductively in a standard way via the Boolean connectives
and quantifiers over set variables as well as over individual variables.

\begin{example}
In MSO of $(\omega; S)$, the following relations and properties are expressible:
 the subset relation $X\subseteq Y$, the natural order
$\leq$, finiteness of sets, and whether a set is a singleton.  For example, 
\[
Single(X) := \exists x (x \in X ~\&~ \forall y (y \in X \iff x =y) ).
\]
\end{example}

The definability of singletons allows us to
transform any MSO formula into an equivalent MSO formula {\em all} of whose
variables are set variables. We can also interpret addition on natural
numbers
as follows. Associate with every finite set $X$ the binary string
$\sigma(X)$ that has $1$ at position $i$ if and only if $i\in X$.  We use the rules of binary addition to express the statement that 
$X$, $Y$, $Z$ are finite sets and $\sigma(X)+_2 \sigma(Y)=\sigma(Z)$ in the MSO logic 
logic of $(\omega; S)$. 

\smallskip

There is a natural correspondence between $k$-ary relations on $\P(\omega)$ (the power set of the natural numbers) and $k$-ary relations on $\{0,1\}^\omega$. For example, consider the case of $k=2$. Any binary relation $R$ on $\P(\omega)$ is a collection of pairs $(X,Y)$ of subsets. Identify $X$ and $Y$ with their characteristic functions $\alpha_X$ and $\alpha_Y$, each of which is an infinite string over $\{0,1\}$.  Recall that the convolution of $(X,Y)$ is an infinite string over $\{0,1\}^2$ such that $c(\alpha_X,\alpha_Y) (i) = (\alpha_X(i), \alpha_Y(i) )$.  The convolution of $R$ is the language $c(R) = \{c(\alpha_X, \alpha_y) : (X,Y) \in R\}$. With this correspondence between relations on $\P(\omega)$ and relations on $\{0,1\}^\omega$, B\"uchi proved a general characterization theorem that links the  MSO definable relations of the successor structure $(\omega; S)$  and B\"uchi automata.

\begin{theorem}[B\"uchi, 1960]\label{thm:Buchi}
A relation $R\subseteq \P(\omega)^k$ is definable in MSO logic if and only
if $R$ is B\"uchi recognizable. Given an MSO definable relation $R$, there is an algorithm which builds a B\"uchi automaton recognizing $R$. In particular, the MSO theory of $(\omega; S)$, denoted as $S1S$, is decidable. \qed
\end{theorem}

We illustrate the theorem on a simple example: checking whether $\exists X \forall Y~ R(X,Y)$ is true in $(\omega;S)$, where $R$ is a definable relation.
Theorem \ref{thm:Buchi} says that $R$ is definable in MSO logic if and only if $c(R)$ 
 is B\"uchi recognizable. Using Theorem \ref{thm:Buchi}, we build an automaton recognizing $R$. This is done 
 inductively based on the complexity of the formula defining $R$. We then construct an
 automaton for  the set of infinite strings $\{ X : \forall Y~ R(X,Y)\}$.   We use Theorem \ref{thm:BucAlgs} again to check if the set $\{ X : \forall Y ~R(X,Y) \}$ is empty.
If it is empty, the sentence is false; otherwise, it is true.  B\"uchi's  theorem is a quintessential example of how understanding the definable relations in a
structure helps us decide the theory of the structure (see Subsection \ref{subsec:Questions}).

\smallskip

We briefly turn our attention to {\bf word automata}.  The underlying graph definition of a word automaton is identical to that of a
B\"uchi automaton.  As mentioned earlier, the difference between these two models of computation lies in 
that word automata process {\em finite} strings rather than infinite
strings.  Let $\M = (S, \iota, \Delta, F)$ be a word automaton over the finite alphabet $\Sigma$, and let $v$ be a finite string over $\Sigma$.  Then $v$ labels
a path in the underlying directed graph of $\M$, starting from the initial state. This path is called a run
of $\M$ on input $v$. If the last state in the run is accepting then the
run is an accepting run. The automaton $\M$ accepts a string if there is some
accepting run of $\M$ on the string. The collection of all finite strings accepted by a
word automaton is called a {\bf regular} (or, equivalently, a FA recognizable) language. As in the setting of B\"uchi automata, we have the following theorem:

\begin{theorem}[Kleene, 1951]\label{thm:wordBA} \hfill
\begin{enumerate}
\item There is an algorithm that, given a word automaton, decides (in
linear time in the size of the automaton) if some string is accepted by the automaton.
\item The collection of all regular languages is closed under the operations of union, intersection, and complementation.
\end{enumerate}
\end{theorem}

\begin{example}The following sets of strings are regular languages.
\begin{itemize}
\item $\{w101: w \in \{0,1\}^\star~\text{has no sub-word of the form}~101\}$
\item $\{w: w \in \{0,1\}^\star$ ~is a binary representation of some positive integer with the least significant bit
first$\}$.
\end{itemize}
\end{example}

\begin{example}
Regular languages can be naturally embedded into B\"uchi recognizable sets.  That is, $W$ is regular if and only if $W \Diamond^{\omega}$ is B\"uchi recognizable (where $\Diamond$ is a new symbol).
\end{example}


Previously, we defined B\"uchi recognizable relations.  We would like an analogous notion for word automata. 
To define regular relations, we need to define the {\bf convolution} of finite strings.  
Let $x_1, \ldots, x_k \in \Sigma^*$.  If $x_1, \ldots, x_k$ are all of the same length, 
the $i^{th}$ element of the convolution $c(x_1, \ldots, x_k)$ is the tuple $\la x_1(i), \ldots, x_k(i) \ra$ (as in the infinite string case).  Otherwise, 
assume without loss of generality that $x_n$ is the longest string among $x_1, \ldots, x_k$.  For each $m\neq n$, append a new symbol $\Diamond$ to the end of $x_m$ as many times as necessary to make the padded version of $x_m$ have the same length as $x_n$.  Then $c(x_1,\ldots, x_k)$ is the convolution of the padded strings.   If $R$ is a $k$-ary relation on $\Sigma^*$, $R$ is called {\bf regular} if its convolution $c(R)$ is a regular language.

\smallskip

As before, we can develop a calculus of regular relations.  Given a regular relation $R$ and regular set $A$, the cylindrification $c(R)$, projection $\exists x_i R$, and universal projection $\forall x_i R$ are all recognizable by finite automaton.  Likewise, given a regular relation $R$ and any finite string $c$, the instantiation $I(R,c)$ is a regular relation.  Similarly, the rearrangement and linkage operations preserve regularity of relations.

\subsection{B\"uchi and word automatic structures}\label{subsec:AutStructs}

The main focus of this tutorial is the study of structures defined by automata. 
We now give a formal definition of this concept and provide several examples of such structures.
Recall that a structure $\A$ is a tuple $(A; R_1, \ldots, R_m)$ where $A$ is a non-empty set called the domain and
$R_1$, $\ldots$, $R_m$ are basic (or atomic) relations on $A$.

\begin{definition}
A structure is {\bf word automatic} if its domain and basic relations are
regular.
A structure is {\bf B\"uchi automatic} if its domain and basic
relations are B\"uchi recognizable.
\end{definition}

Often, we refer to word automatic structures and B\"uchi automatic structures simply as automatic structures. 
The type of automaticity will be clear from the context. We present some examples of word automatic structures.  
We begin with structures whose domains are $\{1\}^\star$ (automatic structures over the one letter alphabet $\{1\}$ are called {\em unary automatic structures}; they have been studied in \cite{BluPhD}, \cite{RubPhD}, \cite{KLM}).

\begin{example}
The structure  $(1^\star; \leq, S)$, where $1^m \leq 1^n \iff m \leq n$ and $S(1^n) = 1^{n+1}$, is  word automatic. 
\end{example}

\begin{example}
The structure $(1^\star ; \mod_1, \mod_2, \ldots, \mod_n)$, where $n$ is a fixed positive integer, is word automatic.   The word automata recognizing the modular relations contain cycles of appropriate lengths.  
\end{example}

Next, we move to structures with a binary alphabet $\{0,1\}$.  It is not too hard to see that any automatic structure over a finite alphabet is isomorphic to an automatic structure over a binary alphabet \cite{RubPhD}.  Clearly, any word automatic structure has a countable domain.

\begin{example}
The structure $(\{0,1\}^\star; \vee, \wedge, \neg)$ is word automatic because bit-wise operations on binary strings can be recognized by finite automata.
\end{example}

\begin{example} Presburger arithmetic, the structure $(\{0,1\}^\star \cdot 1; +_2, \leq)$, where $+_2$ is binary addition if the binary strings are interpreted as the least significant bit first base-$2$ expansion of natural numbers.  The usual algorithm for adding binary  numbers involves a single carry bit, and therefore a small word automaton can recognize the relation $+_2$.   
\end{example}

\begin{example} 
Instead of Presburger arithmetic, we may consider the structure $(\{0,1\}^\star \cdot 1; +_2, \mid_2)$.  This is arithmetic with weak divisibility: $w \mid_2 v$ if $w$ represents a power of $2$ which divides the number represented by $v$.  Since we encode natural numbers by their binary representation, weak divisibility is a regular relation.
\end{example}

\begin{example}
If we treat binary strings at face value rather than as representations of natural numbers, we arrive at a different automatic structure: 
$$
(\{0,1\}^\star; \preceq, Left, Right, EqL),
$$
where $\preceq$ is the prefix relation, $Left$ and $Right$ denote the functions which append a $0$ or $1$ to the binary string (respectively), and $EqL$ is the equal length relation.  It is easy to show that this structure is automatic.  We will see later that this structure has a central role in the study of automatic structures. 
\end{example}

\begin{example}\label{ex:conf}
 A  useful example of a word automatic structure is the configuration space of a 
 Turing machine.  The configuration space is a graph whose nodes are the 
 configurations of the machine (the state, the contents of the tape, and the position of the read/write head).  An edge exists between two nodes if there is a one-step transition of the machine which moves it between the configurations represented by these nodes.  
\end{example}

\begin{example} Any word automatic structure is B\"uchi automatic, but the converse is not true. In the next subsection, we will see examples of B\"uchi automatic structures which have uncountable domains.  These structures cannot be word automatic. 
\end{example}

We mention an automata theoretic version of the theorem of 
L\"owenheim and Skolem. Recall that the classical theorem states that any infinite structure over a countable language
has a countable elementary substructure.  Recall that an elementary substructure is one which has the same first-order theory as the original structure in the language expanded by naming all the elements  of the substructure.  For B\"uchi automatic structures, the analogue of a countable substructure is the substructure consisting of ultimately periodic words.

\begin{theorem}[Hjorth, Khoussainov, Montalb\'an, Nies; 2008]\label{thm:LowSkoForBuchi}
Let $\A$ be a B\"uchi automatic structure and consider the substructure $\A'$ whose domain is
\[
A' = \{ \alpha \in A : \alpha ~\text{is ultimately periodic} \}.
\]
Then $\A'$ is a computable elementary substructure of $\A$.  
Moreover, there is an algorithm which from a first-order formula $\varphi(\bar{a}, \bar{x})$ with parameters $\bar{a}$ from $A'$ produces a B\"uchi automaton which accepts exactly those tuples $\bar{x}$ of ultimately periodic words which make the formula true in the structure.
\end{theorem}
This theorem has also been independently proved by B\'ar\'any and Rubin.  

\subsection{Presentations and operations on automatic structures}\label{sec:AutPresent}

The isomorphism type of a structure is the class of all structures isomorphic to it.
We single out those isomorphism types that contain automatic structures.

\begin{definition}
A structure $\A$ is called (word or B\"uchi) {\bf automata presentable} if it is isomorphic to some (word or B\"uchi) automatic structure $\B$.  In this case, $\B$ is called an {\bf automatic presentation} of $\A$.
\end{definition}

We sometimes abuse terminology and call automata presentable structures automatic.  Let $\B=(D^{\B}; R_1^{\B}, \ldots, R_s^{\B})$ be an automatic presentation of $\A$. Since $\B$ is automatic, the sets $D^{\B}, R_1^{\B}, \ldots, R_s^{\B}$ are all recognized by automata, say by $\M, \M_1,\ldots, \M_s$. Often the automatic presentation of $\A$ is identified with the finite sequence
$\M, \M_1,\ldots, \M_s$ of automata. From this standpoint, automata presentable structures have finite presentations.

\smallskip

Examples of automata presentable structures arise as specific mathematical objects of independent interest, or as the result of closing under automata presentability preserving relations.  
For example, if a class of structures is defined inductively and we know that the base structures have automata presentations and that each of the closure operations on the class preserve automaticity, then we may conclude that each member of the class is automata presentable.  To aid in this strategy, we present some automaticity preserving operations on structures.  The automaton constructions are often slight modifications of those used to show that recognizable sets form a Boolean algebra.  

\begin{proposition}\label{prop:ProdDisjoint}
If $\A$ and $\B$ are (word or B\"uchi) automatic structures then so is their Cartesian product $\A \times \B$.  If $\A$ and $\B$ are (word or B\"uchi) automatic structures then so is their disjoint union $\A + \B$.  
\end{proposition}

\begin{proposition} \label{Prop:quotient}
 If $\A$ is word automatic and $E$ is a regular equivalence relation, then the quotient $\A / E$ is word automatic. 
\end{proposition}
\begin{proof} 
To represent the structure $\A / E$, consider the set
$$Rep=\{x\in A \mid \forall y(y <_{llex} x ~\&~ y\in A \rightarrow (x,y)\not \in E)\},$$ where $<_{llex}$ is the length-lexicographical linear order. 
Recall that for any finite alphabet, the length lexicographic order is defined as $x \leq_{llex} y$ if and only if $|x| < |y|$ or $|x|=|y|$ and $|x| \leq_{lex} |y|$.
Since $Rep$ is first-order definable in $(\A, E)$, it is
regular and constitutes the domain of the quotient structure. Restricting the basic relations of $\A$ to the set $Rep$ yields regular relations. Hence, the quotient structure is automatic. \qed
\end{proof}

A straightforward analysis of countable B\"uchi recognizable languages shows that a countable structure has a word automatic presentation if and only if it has B\"uchi automatic presentation.
It is natural to ask whether Proposition \ref{Prop:quotient} can be extended to countable B\"uchi automatic structures. This has recently been answered positively in \cite{BKR} by a delicate analysis of B\"uchi  recognizable  equivalence relations with countably many equivalence classes.

\begin{theorem}[B\'ar\'any, Kaiser, Rubin; 2008]
For a B\"uchi automatic structure $\A$ and a B\"uchi recognizable equivalence relation $E$ with countably many equivalence classes, the quotient structure  $\A / E$ is 
B\"uchi automatic.
\end{theorem}

A long-standing open question had asked whether B\"uchi automatic structures behave as nicely as word automatic structures with respect to B\"uchi recognizable equivalence relations.  In other words, whether the countability assumption can be removed from the theorem above. A counterexample in \cite{HKMN08} recently settled this question. 
We will outline a proof of the following theorem in the last lecture.

\begin{theorem}[Hjorth, Khoussainov, Montalb\'an, Nies; 2008]\label{thm:BucEquiv}
The class of B\"uchi automatic structures is not closed under the quotient operation with respect to B\"uchi recognizable equivalence relations.
\end{theorem}

We now give some natural examples of automata presentable structures.
 Note that many of the automatic structures that we mentioned earlier arise as the presentations of the following automata presentable structures.  

\begin{example}
The natural numbers under addition and order have a word automata presentation.  The automatic presentation here is the word structure $(\{0,1\}^\star \cdot 1; +_2, \leq)$.
\end{example}
\begin{example}
The real numbers under addition have a B\"uchi automata presentation.  The presentation is $( \{0,1\}^\star \cdot \{ \star\} \cdot \{0,1\}^\omega; +_2)$.
\end{example}
\begin{example}\label{ex:abfg}
Finitely generated abelian groups are all word automata presentable.  Recall that a group $\G$ is finitely generated if there is a finite set $S$ such that $\G$ is the smallest group containing $S$; a group $\G$ is called abelian if the group operation is commutative $a \cdot b = b \cdot a$.  Since every such group is isomorphic to a finite direct sum of $(\Z; +)$ and $(\Z_n; +)$ \cite{Rotman}, and since each of these has an automatic presentation, any finitely generated abelian group has an automatic presentation.
\end{example}
\begin{example}
The Boolean algebra of finite and co-finite subsets of $\omega$ is word automata presentable.  An automatic presentation is the structure whose domain is $\{0,1\}^\star \cup \{2,3\}^\star$ where words in $\{0,1\}^\star$ represent finite sets and words in $\{2,3\}^\star$ represent cofinite sets.
\end{example}
\begin{example}
The Boolean algebra of all subsets of $\omega$ is B\"uchi automata presentable.  A presentation using infinite strings treats each infinite string as the characteristic function of a subset of $\omega$.  The union, intersection, and complementation operations act bitwise on the infinite strings and are recognizable by B\"uchi automata.
\end{example}
\begin{example}
The linear order of  the rational numbers $(\Q; \leq$) has a  word automata presentation: $(\{0,1\}^\star \cdot 1; \leq_{lex})$, where $u \leq_{lex} v$ is the lexicographic ordering and holds if and only if $u$ is a prefix of $v$ or $u = w0x$ and $v=w1y$ for some $w,x,y \in \{0,1\}^\star$.
\end{example}

\subsection{Decidability results for automatic structures}\label{sec:Decide}

B\"uchi's theorem in Subsection \ref{sec:AutDefs} uses automata to prove the decidability of the MSO theory of the successor structure $(\omega; S)$.   We now explore other decidability consequences of algorithms for automata.    The  foundational theorem of Khoussainov and Nerode \cite{KhN95} uses the closure of regular   relations  (respectively, B\"uchi recognizable relations) under Boolean and projection operations to prove the decidability of the first-order theory of any automatic structure.

\begin{theorem}[Khoussainov, Nerode; 1995. \ Blumensath, Gradel; 2000] \label{thm:decidKN}
There is an algorithm that, given a (word or B\"uchi) automatic structure $\A$  and a first-order formula $\varphi(x_1, \ldots, x_n)$, produces an automaton recognizing those tuples $\la a_1, \ldots, a_n \ra$ that make the formula true in $\A$.  
\end{theorem}
\begin{proof}
We go by induction on the complexity of the formula, using the fact that automata recognizable relations are closed under
the Boolean operations and the projection operations as explained in Subsection \ref{sec:AutDefs}. \qed
\end{proof}

The Khoussainov and Nerode decidability theorem can be applied to individual formulas to yield Corollary \ref{cr:KNdecidability1}, or uniformly to yield Corollary \ref{cr:KNdecidability2}.

\begin{corollary}\label{cr:KNdecidability1}
Let $\A$ be a word automatic structure and $\varphi(\bar{x})$ be a first-order formula. There is a linear-time algorithm that, given $\bar{a} \in A$, checks if $\varphi(\bar{a})$ holds in $\A$.
\end{corollary}
\begin{proof}
Let $\A_\varphi$ be an automaton for $\varphi$.    
Given $\bar{a}$, the algorithm runs through the state space of $\A_{\varphi}$ and checks if $\A_{\varphi}$ accepts the tuple. This can be done in linear time in the size of the tuple.   \qed
\end{proof}

\begin{corollary}\label{cr:KNdecidability2}[Hodgson; 1982]
The first-order theory of any automatic structure is decidable.  
\end{corollary}

The connection between automatic structures and first-order formulas goes in the reverse direction as well.  That is, first-order definability can produce new automatic structures.  We say that a structure $\B= (B; R_1, \ldots R_n)$ is first-order definable in a structure $\A$ if there are first-order formulas  $\varphi_B$ and $\varphi_1, \ldots, \varphi_n$ (with parameters from $\A$) which define $B, R_1, \ldots, R_n$ (respectively) in the structure $\A$.  Khoussainov and Nerode's theorem immediately gives the following result about first-order definable structures.

\begin{corollary}\label{cr:Definability}
If $\A$ is (word or B\"uchi) automatic and $\B$ is first-order definable in $\A$, then $\B$ is also (word or B\"uchi) automatic.  
\end{corollary}

In fact, automatic structures can yield algorithms for properties expressed in logics stronger than first-order.  We denote by $\exists^{\infty}$ the \lqt there are countably infinitely many\rqt quantifer, and by $\exists^{n,m}$ the \lqt there are $m$ many mod $n$\rqt quantifiers.  Then ($FO + \exists^{\infty} + \exists^{n,m}$) is the first-order logic extended with these quantifiers.  The following theorem from \cite{KhRS04} extends the Khoussainov and Nerode theorem to this extended logic.  We note that Blumensath noted the $\exists^{\infty}$ case first, in \cite{BluPhD}.

\begin{theorem}[Khoussainov, Rubin, Stephan; 2003]\label{thm:decidKRS}
There is an algorithm that, given a word automatic structure $\A$  and a ($FO + \exists^{\infty} + \exists^{n,m}$) formula $\varphi(x_1, \ldots, x_n)$ with parameters from $\A$, produces an automaton recognizing those tuples $\la a_1, \ldots, a_n\ra$ that make the formula true in $\A$.  
\end{theorem}

\begin{corollary}
The ($FO + \exists^{\infty} + \exists^{n,m}$)-theory of any word automatic structure is decidable.
\end{corollary}

The next corollary demonstrates how the extended logic can be used in a straightforward way.

\begin{corollary}
If $L$ is a word automatic partially ordered set, the set of all pairs $\la x,y\ra$ such that the interval $[x,y]$ has an even number of elements is regular.
\end{corollary}

Another interesting application is an automata theoretic version of K\"onig's lemma.
Recall that the classical version of K\"onig's Lemma says that every infinite finitely branching tree contains an infinite path.  In Lecture 2, we discuss K\"onig's lemma in greater detail.

\begin{corollary}\label{cr:autKL1}
Let $\T=(T; \leq)$ be a word automatic infinite finitely branching tree. There exists an infinite regular path in $\T$.
\end{corollary}

\begin{proof}
We assume that the order $\leq$ is such that the root of the tree is the least element.
We use the auxiliary automatic relations $\leq_{llex}$ (the length lexicographic order) to give a $(FO + \exists^\infty)$ definition of an infinite path of $\T$.  Recall that the length lexicographic order is defined as $x \leq_{llex} y$ if and only if $|x| < |y|$ or $|x|=|y|$ and $|x| \leq_{lex} |y|$.
\[
P = \{ x : \exists^{\infty} y  (x \leq y) ~\&~\forall (y \leq x) \forall (z\neq z' \in S(y) ) [z \leq x \implies z <_{llex} z'] \}
\]
In words, $P$ is the left-most infinite path in the tree.  The first clause of the definition restricts our attention to those nodes of the tree which have infinitely many descendants.  Since $\T$ is finitely branching, these nodes are exactly those which lie on infinite paths. It is easy to see that the above definition guarantees that $P$ is closed downward, linearly ordered, and finite.  Moreover, it is regular by the decidability theorem, Theorem \ref{thm:decidKN}. \qed
\end{proof}

Recently, Kuske and Lohrey generalized the decidability theorem for ($FO + \exists^{\infty} + \exists^{n,m}$) to B\"uchi automatic structures \cite{KL06}.

\subsection{Definability in automatic structures}\label{sec:CompAutStr}

Throughout computer science and logic, there are classifications of problems or sets into hierarchies.  Some examples include time complexity, relative computability, and proof-theoretic strength.  In each case, there is a notion of reducibility between members in the hierarchy.  One often searches for a complete, or typical, member at each level of the hierarchy.  More precisely, an element of the hierarchy is called complete for a particular level if it is in that level, and if all other elements of the level are reducible to it.  We may view automatic structures as a complexity class.  In that context, we would like to find complete structures.  For this question to be well-defined, we must specify a notion of reducibility.  
In light of the results of the previous section, it seems natural to consider logical definability of structures as the notion of reducibility. For B\"uchi automatic structures, B\"uchi's theorem immediately gives a complete structure with respect to MSO definability.

\begin{corollary}
A structure is B\"uchi automatic if and only if it is definable in the MSO logic of the successor structure $(\omega; S)$.
\end{corollary}

For word automatic structures, it turns out that first-order definability suffices.  Blumensath and Gr\"adel identify the following complete structure \cite{BluGr00}.  (Each of the basic relations in the following structure is defined in Subsection \ref{sec:AutDefs}.)

\begin{theorem}[Blumensath, Gr\"adel; 2000]\label{thm:defin}
A structure is word automatic if and only if it is 
first-order definable in the word structure $(\{0,1\}^\star; \preceq, Left, Right, EqL)$.
\end{theorem}

\begin{proof}
One direction is clear because automatic structure are closed under first-order definability.  For the converse, suppose that $\A$ is an automatic structure.  We will show that it is definable in $(\{0,1\}^\star; \preceq, Left, Right, EqL)$.  

\smallskip

It suffices to show that every regular relation on $\{0,1\}^\star$ is definable in the word structure. Without loss of generality, assume that $L$ is a language recognized by the
word automaton  $\M = (S, \iota, \Delta, F)$.  Assume $S = \{1, \ldots, n \}$ with $1 = \iota$ (the initial state).  We define a formula $\varphi_\M(v)$ in the language of $(\{0,1\}^\star; \preceq, Left, Right, EqL)$ which will hold of the string $v$ if and only if $v$ is accepted by $\M$.   We use the following auxiliary definable relations in our definition of $\varphi_\M$.
\begin{itemize}
\item Length order, $|p| \leq |x|$, is defined by $\exists y ( y \preceq x ~\&~EqL(y,p) )$);
\item The digit test relation, the digit of $x$ at position $|p|$ is $0$, is defined by 
$\exists y\exists z (z \preceq y \preceq x ~\&~ EqL(y,p) ~\&~ Left(z,y) )$);
\item The distinct digits relation states that the digits of $x_1$ and $x_2$ at position $|p|$ are distinct.
\end{itemize}
The following paragraph describing a run of $\M$ on input $v$ can now be translated into a single first-order sentence $\varphi_\M$ with parameter $v$. \lqt There are strings $x_1, \ldots, x_n$ each of length $|v|+1$.  For each $p$, if $|p| \leq |v|+1$, there is exactly one $x_j$ with digit $1$ at position $|p|$ (the strings $x_1, \ldots, x_n$ describe which states we're in at a given position in the run). If $x_i$ has digit $1$ at position $|p|$, and $\sigma$ is the digit of $v$ at position $|p|$, and $\Delta(i,\sigma) = j$ then $x_j$ has digit $1$ at positions $|p|+1$.  The digit in the first position of $x_1$ is $1$.  There is some $x_j$ such that $j \in F$ and for which the digit in the last position is $1$.\rqt  Therefore, $L$ is first-order definable in $(\{0,1\}^\star; \preceq, Left, Right, EqL)$.  Since the domain and each of the relations of an automatic structure are regular, they are also first-order definable. \qed
\end{proof}

If we use weak monadic second-order, WMSO, logic (where set quantification is only over finite sets) instead of first-order logic as our notion of reducibility, we arrive at a more natural complete structure.  Blumensath and Gr\"adel \cite{BluGr00} show that, in this case, the successor structure is complete.  This result has nice symmetry with the B\"uchi case, where a structure is B\"uchi recognizable if and only if it is MSO definable in $(\omega; S)$.

\begin{corollary}
A structure is word automatic if and only if it is WMSO definable in the successor structure $(\omega; S)$.
\end{corollary}

\begin{proof}
By Theorem \ref{thm:defin}, it suffices to give a WMSO definition of the structure $(\{0,1\}^\star; \preceq, Left, Right, EqL)$ in $(\omega; S)$.  To do so, interpret each $v \in \{0,1\}^\star$ by the set $Rep(v) = \{ i : v(i) = 1 \}  \cup \{ |v|+1 \}$.  Then $Rep(v)$ is a finite set and for each nonempty finite set $X$ there is a string $v$ such that $Rep(v)= X$.  Moreover, under this representation, each of $Left, Right, \preceq, EqL$ is a definable predicate. \qed
\end{proof}

We refer the reader to \cite{BluGr00}, \cite{RubPhD} ,\cite{RubinSurvey}  for issues related to definability in automatic structures. 


\section{
{\em Characterization Results}}


\subsection{Automata on trees and tree automatic structures}\label{sec:TreeAut}

The two flavours of automata we have presented so far operate on linear inputs: finite and infinite strings.  We now take a slight detour and consider labelled trees as inputs for automata.  Our archetypal tree is the two successor structure,
$$
\T = (\{ 0,1 \}^\star; Left, Right).
$$  
The root of this binary tree is the empty string, denoted as $\lambda$.  Paths in $\T$ are defined by infinite strings in $\{0,1\}^\omega$.  Let $\Sigma$ be a finite alphabet.  A $\Sigma$-labelled tree $(\T,v)$ associates a mapping $v : \T \to \Sigma$ to the binary tree.  The set of all $\Sigma$-labelled trees is denoted by $Tree(\Sigma)$.  

\smallskip

A {\bf Rabin automaton} $\M$ is specified by $\M = (S, \iota, \Delta, \F)$, where $S$ and $\iota$ are the finite set of states and the initial state, the transition 
relation is $\Delta \subset S \times \Sigma \times (S \times S)$,
and the accepting condition is given by $\F \subset \P(S)$. An input to a Rabin automaton is a labelled tree $(\T, v)$.  A run of $\M$ on $(\T, v)$ is a mapping $r: \T \to S$ which respects the transition relation in that
\[
r(\lambda) = \iota, \qquad
 \forall x \in \T\big[\big(r(x), v(x), r (Left(x)), r(Right(x)) \big) \in \Delta \big].
\]
The run $r$ is accepting if for every path $\eta$ in $\T$, the set 
\[
In(\eta) = \{s \in S : s ~\text{appears on}~ \eta~ \text{infinitely many times} \}
\]
 is an element of $\F$.  The language of a Rabin automaton $\M$, denoted as $L(\M)$, is the set of all $\Sigma$-labelled trees $(\T, v)$ accepted by $\M$.  

\begin{example}
Here are a few examples of Rabin automata recognizable sets of $\{0,1\}$-labelled trees.
\begin{itemize}
\item $\{ (\T, v) : v(x) = 1 ~\text{for only finitely many $x \in \T$} \}$
\item $\{ (\T, v) : \text{each path has infinitely many nodes labelled $1$} \}$
\item $\{ (\T, v) : \forall x \in \T ( v(x) = 1 \implies \text{the subtree rooted at $x$ is labelled by $0$s only}) \}$
\item $\{ (\T, v) : \exists x \in \T (v(x) = 1) \}$.
\end{itemize}
\end{example}

As you may recall from Subsection \ref{sec:AutDefs}, fundamental facts about B\"uchi automata gave us algorithms for checking emptiness and for constructing new automata from old ones.  Rabin's breakthrough theorems in \cite{Rab69} yield analogous results for Rabin automata.

\begin{theorem}[Rabin; 1969] \hfill
\begin{enumerate}
\item There is an algorithm that, given a Rabin automaton $\M$, decides if $L(\M)$ is empty.
\item The class of all Rabin recognizable tree languages is effectively closed under the operations of union, intersection, and complementation.
\end{enumerate}
\end{theorem}

In the setting of sequential automata, B\"uchi's theorem (Theorem \ref{thm:Buchi}) connected automata and logic.   In particular, the logic used was MSO, where we allow quantification both over elements of the domain and over subsets of the domain.  MSO on the binary tree $\T$ can be used to express properties of sets of trees such as $Path(X), Open(X), Clopen(X)$.   Rabin's theorem \cite{Rab69} connects MSO definability and automaton recognizability.  Note that, as in the B\"uchi case, convolutions can be used to define Rabin recognizability of relations on trees.

\begin{theorem}[Rabin; 1969] 
A relation $R \subseteq \P (\T)^k$ is definable in the MSO logic of the two successor structure if and only if $R$ is recognizable by a Rabin automaton.  In particular, the MSO theory of $\T$, denoted as S2S, is decidable.
\end{theorem}

This theorem has led to numerous applications in logic and theoretical computer science.  Many of these applications involve proving the decidability of a particular theory by reducing it to the MSO theory of the binary tree.

\smallskip

Rabin automata have natural counterparts which work on finite binary trees.   Let $L(X)$ denote the set of leaves (terminal nodes) of a finite binary tree $X$. A finite $\Sigma$-tree is a pair $(X, v)$ where $X$ is a finite binary tree and $v : X \setminus L(X) \to \Sigma$ is a mapping which labels non-leaf nodes in $X$ with elements of the alphabet.  
A (top-down) {\bf tree automaton} is $\M = (S, \iota, \Delta, F)$ where $S, \iota, F$ are as in the word automatic case, and 
$\Delta \subset S \times \Sigma \times (S \times S)$
is the transition relation.  A run of $\M$ on a finite $\Sigma$-tree $(X,v)$ is a map $r: X \to S$ which satisfies $r(\lambda) = \iota$ and 
$\big(r(x), v(x), r(Left(x)), r(Right(x))\big) \in \Delta$.
 The run is accepting if each leaf node $x \in L(x)$ is associated to an accepting state $r(x) \in F$.  The language of a tree automaton is the set of finite $\Sigma$-trees it accepts.  As before, tree automata have pleasant algorithmic properties (see the discussions in \cite{Doner65}, \cite{TW65}, \cite{Rab69}).

\begin{theorem}[Doner; 1965.  Thatcher, Wright; 1965.  Rabin; 1969]\hfill
\begin{enumerate}
\item There is an algorithm that, given a tree automaton $\M$, decides if $L(\M)$ is empty.
\item The class of all tree automata recognizable languages is effectively closed under the operations of union, intersection, and complementation.
\end{enumerate}
\end{theorem}

Automata which work on trees can be used to define tree automatic structures.  In this context, domain elements of structures are represented as trees rather than strings.

\begin{definition}
A structure is {\bf tree automatic} if its domain and basic relations are recognizable by tree automata.
A structure is {\bf Rabin automatic} if its domain and basic relations are recognizable by Rabin automata.
\end{definition}

Every tree automatic structure is Rabin automatic (we can pad finite trees into infinite ones).  Since strings embed into trees, it is easy to see that every word automatic structure is tree automatic and that every B\"uchi automatic structure is Rabin automatic.  However, this inclusion is strict.  For example, the natural numbers under multiplication $(\omega; \times)$ is a tree automatic structure but (as we will see in the next subsection) it is not word automatic.  Similarly, the countable atomless Boolean algebra is tree automatic but not word automatic \cite{KhNieRuSte04}.  In Lecture 3, we will discuss a recent result which separates B\"uchi and Rabin structures.
Although there is a strict separation between the classes of sequential-input and branching-input automatic structures, their behaviour in terms of definability is very similar.  

\begin{theorem}[Rabin; 1969]
A structure is Rabin automatic if and only if it is MSO definable in the binary tree $\T$.
\end{theorem}

\begin{theorem}[Rabin; 1969]
A structure is tree automatic if and only if it is WMSO definable in the  binary tree $\T$.
\end{theorem}

Similarly, we have a L\"owenheim-Skolem Theorem for Rabin automatic structures akin to Theorem \ref{thm:LowSkoForBuchi} for B\"uchi automatic structures.  Recall that a $\Sigma$-labelled tree is called {\bf regular} if it has only finitely many isomorphic sub-trees.

\begin{theorem}[Hjorth, Khoussainov, Montalb\'an, Nies; 2008]
Let $\A$ be a Rabin automatic structure and consider the substructure $\A'$ whose domain is
\[
\A' = \{ \alpha \in A : \alpha ~\text{is a regular tree} \}.
\]
Then $\A'$ is a computable elementary substructure of $\A$.  
Moreover, there is an algorithm that from a first-order formula $\varphi(\bar{a}, \bar{x})$ with parameters $\bar{a}$ from $A'$ produces a Rabin automaton which accepts exactly those regular tree tuples $\bar{x}$ which make the formula true in the structure.
\end{theorem}

\subsection{Proving non-automaticity}\label{sec:NonAut}

Thus far, our toolbox contains several ways to prove that a given structure is automatic (explicitly exhibiting the automata, using extended first-order definitions, and using interpretations in the complete automatic structures).  However, the only proof we've seen so far of non-automaticity is restricted to word automata and uses cardinality considerations (cf. Subsections \ref{sec:AutDefs}, \ref{sec:AutPresent}).  We could also use general properties of automatic structures to prove that a given structure is not automatic; for example, if a structure has undecidable first-order theory then it is not automatic by Corollary \ref{cr:KNdecidability2}.  We will now see a more careful approach to proving non-automaticity.

\smallskip

In finite automata theory, the Pumping Lemma is a basic tool for showing non-regularity.  Recall that the lemma says that if a set $L$ is regular then there is some $n$ so that for every $w \in L$ with $|w| > n$, there are strings $x,u,v$ with $|u|>0$ such that $w = xuv$ and for all $m$, $xu^m v \in L$.  The constant $n$ is the number of states of the automaton recognizing $L$.  The following Constant Growth Lemma uses the Pumping Lemma to arrive at an analogue for automatic structures \cite{KhN95}. 

\begin{lemma}[Khoussainov, Nerode; 1995]\label{lm:ConstantGrowth}
If $f: D^n \to D$ is a function whose graph is a regular relation, there is a constant $C$ (which is the number of states of the automaton recognizing the graph of $f$) such that for all $x_1, \ldots, x_n \in D$
\[
|f(x_1, \ldots, x_n)| \leq \max \{ |x_1|, \ldots, |x_n| \} + C.
\]
\end{lemma}
\begin{proof}
Suppose for a contradiction that $|f(x_1, \ldots, x_n)| - \max \{ |x_1|, \ldots, |x_n| \} > C$.  Therefore, the convolution of the tuple $(x_1, \ldots, x_n, f(x_1, \ldots, x_n) )$ contains more than $C$ $\diamond$'s appended to each $x_i$. In particular, some state in the automaton is visited more than once after all the $x_i$'s have been read.  As in the Pumping Lemma, we can use this to obtain infinitely many tuples of the form $(x_1, \ldots, x_n, y)$ with $y \neq f(x_1, \ldots x_n)$ accepted by the automaton, or it must be the case that the automaton accepts strings which do not represent convolutions of tuples.  Both of these cases contradict our assumption that the language of the automaton is the graph of the function $f$. \qed
\end{proof}

The Constant Growth Lemma can be applied in the settings of automatic monoids and automatic structures in general to give conditions on automaticity \cite{KhNieRuSte04}.  Recall that a monoid is a structure $(M; \cdot)$ whose binary operation $\cdot$ is associative.

\begin{lemma}[Khoussainov, Nies, Rubin, Stephan; 2004]\label{lm:GenMonoids}
If $(M; \cdot)$ is an automatic monoid, there is a constant $C$ (the number of states in the automaton recognizing $\cdot$) such that for every $n$ and every $s_1, \ldots, s_n \in M$
\[
|s_1 \cdot s_2 \cdot  \cdots \cdot s_n | \leq \max \{ |s_1|, |s_2|, \ldots, |s_n| \} + C \cdot \lceil \log (n) \rceil.
\]
\end{lemma}

\begin{proof}
Let $C$ be the number of the states in the automaton recognizing the graph of the monoid multiplication.  We proceed by induction on $n$.  In the base case, the inequality is trivial: $|s_1| \leq |s_1|$.  For $n > 1$, write $n = u+v$ such that $u= \lfloor \frac{n}{2} \rfloor$.  Note that $u < n$ and $v < n$.  Let $x_1 = s_1 \cdot \cdots \cdot s_u$ and $x_2 = s_{u+1} \cdot \cdots \cdot s_n$.  By the induction hypothesis, $|x_1|\leq \max \{ |s_1|, \ldots, |s_u| \} + C \cdot \lceil \log(u)\rceil$ and $|x_2|\leq \max \{ |s_{u+1}|, \ldots, |s_n| \} + C \cdot \lceil \log(v)\rceil$.  Applying Lemma \ref{lm:ConstantGrowth}, 
\begin{align*}
|s_1 \cdot \cdots \cdot s_n| &= |x_1 \cdot x_2| \leq \max\{|x_1|, |x_2|\} + C \\
&\leq \max\{ |s_1|, \ldots, |s_n| \} + C \max \{ \lceil \log(u) \rceil, \lceil \log(v)\rceil \} + C\\
&\leq \max\{ |s_1|, \ldots, |s_n| \} + C \lceil \log(n) \rceil.
\end{align*}\qed
\end{proof}

Let $\A = (A; F_0, F_1, \ldots, F_n)$ be an automatic structure.   Let $X =\{ x_1, x_2, \ldots\}$ be  a subset of $A$.  The generations of $X$ are the elements of $A$ which can be obtained from $X$ by repeated applications of the functions of $\A$.  More precisely, 
\[
G_1(X) = \{ x_1\}~~\text{and}~~G_{n+1} (X) = G_n(X) \cup \{ F_i (\bar{a} ) : \bar{a} \in G_n(X) \} \cup \{ x_{n+1} \}.  
\]
The Constant Growth Lemma  dictates the rate at which generations of $X$ grow and yields the following theorem.

\begin{theorem}[Khoussainov, Nerode; 1995.  Blumensath; 1999]\label{thm:GrowthGen}
Suppose $X \subset A$ and there is a constant $C_1$ so that in the length lexicographic listing of $X$ ($x_1 <_{llex} x_2 <_{llex} \cdots$) we have $|x_n| \leq C_1 \cdot n$ for all $n\geq 1$.  Then there is a constant $C$ such that $|y| \leq C \cdot n$ for all $y \in G_n(X)$.  In particular, 
\[
G_n(X) \subseteq \Sigma^{ \leq C \cdot n} 
\]
if $|\Sigma| > 1$, and $|G_n(X)| \leq C \cdot n$ if $|\Sigma| = 1$.
\end{theorem}

Just as the Pumping Lemma allows immediate identification of certain non-regular sets, the above theorem lets us determine that certain structures are not automatic.  

\begin{corollary}
The free semigroup $(\{0,1\}^\star; \cdot)$ is not word automatic.  Similarly, the free group $F(n)$ with $n > 1$ generators is not word automatic.
\end{corollary}
\begin{proof}
We give the proof for the free semigroup with two generators.  Consider $X = \{0,1\}$. By induction, we see that for each $n$, $\{0,1\}^{< 2^n} \subseteq G_{n+1} (\{0,1\})$.  Therefore, $|G_{n+1}(X)| \geq 2^{2^n - 1} - 1$ and hence can't be bounded by $2^{C \cdot n}$ for any constant $C$. \qed
\end{proof}
Similarly, one can prove the following.
\begin{corollary}
For any bijection $f: \omega\times \omega \to \omega$, the structure $(\omega; f)$ is not word automatic.
\end{corollary}
The proofs of the next two corollaries require a little more work but employ growth arguments as above.
\begin{corollary}
The natural numbers under multiplication $(\omega; \times)$ is not word automatic.
\end{corollary}

\begin{corollary}
The structure $(\omega; \leq, \{ n! : n \in \omega\})$, where the added unary predicate picks out the factorials, is not word automatic.
\end{corollary}

We now switch our focus to subclasses of (word) automatic structures.  For some of these classes, structure theorems have been proved which lead to good decision methods for questions like the isomorphism problem.  Such structure theorems must classify both the members of a class and the non-members.  Hence, techniques for proving non-automaticity become very useful.  In other cases, as we will see in the next lecture, complexity results give evidence that no nice structure theorems exists.  The classes we consider below are partial orders, linear orders, trees, Boolean algebras, and finitely generated groups.

\subsection{Word automatic partial orders}\label{sec:AutPo}

The structure $(A; \leq)$ is a {\bf partially ordered set} if the binary relation $\leq$ is reflexive $\forall x (x \leq x)$, anti-symmetric $\forall x\forall y (x \leq y ~\&~ y \leq x \rightarrow x=y)$, and transitive $\forall x \forall y \forall z (x \leq y ~\&~ y \leq z \rightarrow x \leq z$).  A partially ordered set $(A;  \leq)$ is word automatic if and only if $A$ and $\leq$ are both recognized by word automata.  For the rest of this section, we deal only with word automatic partial orders and for brevity, we simply call them automatic. 

\begin{example}
We have already seen some examples of automatic partial orders: the full binary tree under prefix order $(\{0,1 \}^\star; \preceq)$; the finite and co-finite subsets of $\omega$ under subset inclusion; the linear order of the rational numbers.
\end{example}

Recall that a linear order is one where $\leq$ is also total: for any $x,y$ in the domain, either $x \leq y$ or $y \leq x$. 

\begin{example}\label{ex:ordinals}
Small ordinals (such as $\omega^n$ for $n$ finite) \cite{KhN95} are automatic partial orders.  In fact, Delhomm\'e showed that automatic ordinals are exactly all those ordinals below $\omega^\omega$ \cite{Del04}. We will see the proof of this fact later in this subsection.
\end{example}

\begin{example}
The following example was one of the first examples of a non-trivial automatic linear order \cite{RubPhD}. Given an automatic linear order $L$ and a polynomial $f(x)$ with positive integer coefficients, consider the linear order
\[
\Sigma_{x \in \omega} (L + f(x))
\] 
where we have a copy of $L$, followed by a finite linear order of length $f(0)$, followed by another copy of $L$, followed by a finite linear order of length $f(1)$, etc.  This linear order is automatic.  In fact, the linear order obtained by the same procedure where the function $f$ is an exponential $f(x) = a^{b \cdot x + c}$ with $a,b,c \in \omega$ is also automatic.  \end{example}

The last example involves the addition of linear orders.  The sum of orders $L_1$ and $L_2$ is the linear order in which we lay down the order $L_1$ and then we place all of $L_2$.   Since $L_1 + L_2$ is first-order definable from the disjoint union of $L_1$ and $L_2$,  Proposition \ref{prop:ProdDisjoint} and Corollary \ref{cr:Definability} imply that the sum operation preserves automaticity. Another basic operation on linear orders also preserves automaticity: the product linear order $L_1 \cdot L_2$ is one where a copy of $L_1$ is associated with each element of $L_2$; each of the copies of $L_1$ is ordered as in $L_1$, while the order of the copies is determined by $L_2$. The order  $L_1 \cdot L_2$ is also first-order definable from the disjoint union of $L_1$ and $L_2$, and hence is automatic if $L_1$ and $L_2$ are automatic.

\smallskip

We use several approaches to study the class of automatic partial orders.  First, we restrict to well-founded partial orders and consider their ordinal heights.  We will see that automatic partial orders are exactly those partial orders with relatively low ordinal heights.  This observation parallels Delhomm\'e's previously mentioned result that automatic ordinals are exactly those below $\omega^\omega$ \cite{Del04}.  Next, we study automatic linear orders.  We present results about the Cantor-Bendixson ranks of automatic linear orders, and see the implications of these results for decidability questions.  In particular, we see that the isomorphism problem for automatic ordinals is decidable. 
 Finally, we consider partial orders as trees and consider the branching complexity of automatic partial order trees.  We present several automatic versions of K\"onig's famous lemma about infinite trees.

\smallskip

We now introduce well-founded partial orders and ordinal heights.   A binary relation $R$ is called {\bf well-founded} if there is no infinite chain of elements $x_0, x_1, x_2, \ldots$ such that $(x_{i+1}, x_i) \in R$ for all $i$.  For example, $(\Z^+; S)$ is a well-founded relation but $(\Z^-; S)$ is not well-founded (we use $\Z^+$ and $\Z^-$ to denote the positive and negative natural numbers).  Given a well-founded structure $\A = (A;R)$ with domain $A$ and binary relation $R$, a {\bf ranking function} for $\A$ is an ordinal-valued function $f$ on $A$ such that $f(y) < f(x)$ whenever $(y,x) \in R$.  We define $ord(f)$ as the least ordinal larger than or equal to all values of $f$. It is not hard to see that 
$\A = (A;R)$ is well-founded if and only if $\A$ has a ranking function.

\smallskip

Given a well-founded structure $\A$, its {\bf ordinal height} (denoted $r(\A)$) is the least ordinal $\alpha$ which is $ord(g)$ for some ranking function $g$ for $\A$.  An equivalent definition of the ordinal height uses an assignment of rank to each element in the domain of $\A$.  If $x$ is an $R$-minimal element of $A$, set $r_{\A} (x) = 0$.  For any other element in $A$, put $r_{\A}(z) = \sup\{ r(y) + 1 : (y,z) \in R \}$.  Then, we define $r(\A) = \sup \{r_{\A}(x) : x \in A \}$.    The following property of ordinal heights is useful when we work with substructures of well-founded relations.

\begin{lemma}\label{lm:wfElementRank}
Given a well-founded structure $\A = (A, R)$, if $r(\A) = \alpha$ and $\beta < \alpha$ then there is $x \in A$ such that $r_{\A}(x) = \beta$.
\end{lemma}

The ordinal height can be used to measure the depth of a structure.  In our exploration of automatic structures, we study the ordinal heights attained by automatic well-founded relations.  As a point of departure, recall that any automatic structure is also a computable structure (see Subsection \ref{subsec:history}).  We therefore begin by considering the ordinal heights of computable structures.  An ordinal is called computable if it is the order-type of some computable well-ordering of the natural numbers.

\begin{lemma}\label{lm:wfComputable}
Each computable ordinal is the ordinal height of some computable well-founded relation.  Conversely, the ordinal height of each computable well-founded relation is a computable ordinal.
\end{lemma}

Since any automatic structure is a computable structure, Lemma \ref{lm:wfComputable} gives us an upper bound on the ordinal heights of automatic well-founded relations.  We now ask whether this upper bound is sharp.  We will consider this question both in the setting of all automatic well-founded relations (in Lecture 3), and in the setting of automatic well-founded partial orders (now).  The following theorem characterizes automatic well-founded partial orders in terms of their ordinal heights.

\begin{theorem}[Khoussainov, Minnes; 2007]\label{thm:wfpo}
An ordinal $\alpha$ is the ordinal height of an automatic well-founded partial order if and only if $\alpha < \omega^\omega$.  
\end{theorem}

One direction of the proof of the characterization theorem is easy: each ordinal below $\omega^\omega$ is automatic (Example \ref{ex:ordinals}) and is an automatic well-founded total order.  Moreover, the ordinal height of an ordinal is itself.  

\smallskip

For the converse, we will use a property of the natural sum of ordinals.  The {\bf natural sum} of $\alpha$ and $\beta$, denoted $\alpha +' \beta$, is defined recursively as $\alpha +' 0 = \alpha$, $0 +' \beta = \beta$, and $\alpha +' \beta$ is the least ordinal strictly greater than $\gamma +' \beta$ for all $\gamma < \alpha$ and strictly greater than $\alpha +' \gamma$ for all $\gamma < \beta$.  An equivalent definition of the natural sum uses the {\bf Cantor normal form} of ordinals.  Recall that any ordinal can be written in this normal form as 
\[
\alpha = \omega^{\beta_1} n_1 + \omega^{\beta_2} n_2 + \cdots + \omega^{\beta_k} n_k,
\] 
where $\beta_1 > \beta_2 > \cdots > \beta_k$ and $k, n_1, \ldots, n_k \in \N$. We define
\[
(\omega^{\beta_1} a_1 + \cdots + \omega^{\beta_k} a_k) +' (\omega^{\beta_1} b_1 + \cdots + \omega^{\beta_k} b_k)  = \omega^{\beta_1} (a_1+b_1) + \cdots + \omega^{\beta_k} (a_k+b_k). 
\]
The following lemma gives sub-additivity of ordinal heights of substructures with respect to the natural sum of ordinals.
\begin{lemma}\label{lm:RankLemma}
Suppose $\A = (A; \leq)$ is a well-founded partial order and $A_1, A_2$ form a partition of $A$ ($A_1 \sqcup A_2 = A$, a disjoint union).  Let $\A_1 = (A_1; \leq_1), \A_2 = (A_1; \leq_2)$ be obtained by restricting $\leq$ to $A_1, A_2$.  Then $r(\A) \leq r(\A_1) +' r(\A_2)$.
\end{lemma}
\begin{proof}
For each $x \in A$, consider the sets $A_{1,x} = \{ z \in A_1 : z < x \}$ and $A_{2,x} = \{ z \in A_2 : z < x \}$.  The structures $\A_{1,x}, \A_{2,x}$ are substructures of $\A_1, \A_2$ respectively. Define a ranking function of  $\A$ by $f(x) = r(\A_{1,x}) +' r(\A_{2,x})$.  The range of $f$ is contained in $r(\A_1) +' r(\A_2)$.  Therefore, $r(\A) \leq r(\A_1) +' r(\A_2)$. \qed
\end{proof}

We now outline the proof of the non-trivial direction of the characterization theorem of automatic well-founded partial orders.  Note that this proof follows Delhomm\'e's proof that ordinals larger than $\omega^\omega$ are not automatic \cite{Del04}.  We assume for a contradiction that there is an automatic well-founded partial order $\A = (A; \leq)$ such that $r(\A) = \alpha \geq \omega^\omega$.  Let $\M_A = (S_A, \iota_A, \Delta_A, F_A)$, $\M_{\leq} = (S_\leq, \iota_\leq, \Delta_\leq, F_\leq)$ be the word automata which recognize $A$ and $\leq$ (respectively).  For each $u \in A$, the set of predecessors of $u$, denoted by $u \downarrow$, can be partitioned into finitely many disjoint pieces as
\[
u \downarrow = \{ x \in A: |x| < |u| ~\&~ x < u \} ~\sqcup~ \bigsqcup_{v \in \Sigma^\star: |v|=|u|} X_v^u
\]
where $X_v^u = \{ vw \in A : vw < u \}$ (extensions of $v$ which are predecessors of $u$).
Since $r(\A) \geq \omega^\omega$, Lemma \ref{lm:wfElementRank} guarantees that for each $n$, there is an element $u_n \in A$ with $r_{\A}(u) = \omega^{n}$.  Moreover, Lemma \ref{lm:RankLemma} implies that if a structure has ordinal height $\omega^n$, any finite partition of the structure contains a set of ordinal height $\omega^n$.  In particular, for each $u_n$ there is $v_n$ such that $|u_n| = |v_n|$ and $r(X_{v_n}^{u_n}) = r(u_n) \downarrow = \omega^n$.  We now use the automata $\M_A, \M_\leq$ to define an equivalence relation of finite index on pairs $(u,v)$: $(u,v) \sim (u',v')$ if and only if $\Delta_A(\iota_A, v) = \Delta_A(\iota_A, v')$ and $\Delta_{\leq}(\iota_\leq, \binom{v}{u}) = \Delta_{\leq}(\iota_\leq, \binom{v'}{u'})$.  Suppose that $(u,v) \sim (u',v')$.  Then the map $f: X_v^u \to X_{v'}^{u'}$ defined as $f(vw) = v'w$ is an order-isomorphism.  Hence, $r(X_v^u) = r(X_{v'}^{u'})$.  Also, there are at most $|S_A| \times |S_\leq|$ $\sim$equivalence classes.  Therefore, the sequence $\{(u_n, v_n)\}$ contains some $m,n$ $(m \neq n)$ such that $(u_m, v_m) \sim (u_n, v_n)$.  But, $\omega^m = r(u_m) = r(X_{v_m}^{u_m}) = r(X_{v_n}^{u_n}) = r(u_n) = \omega^n$.  This is a contradiction with $m \neq n$.  Thus, there is no automatic well-founded partial order whose ordinal height is greater than or equal to $\omega^\omega$. \qed

\smallskip

The above characterization theorem applies to automatic {\em well-founded} partial orders.  We now examine a different class of automatic partial orders: linear orders. 
We seek a similar characterization theorem for automatic linear orders based on an alternate measure of complexity.  Let $(L; \leq)$ be a linear order.  Then $x,y \in L$ are called $\equiv_F$-equivalent if there are only finitely many elements between them.  We can use $\equiv_F$ equivalence to measure how far a linear order is from \lqt nice\rqt linear orders like $(\omega; \leq)$ or $(\Q; \leq)$.  To do so, given a linear order, we take its quotient with respect to the $\equiv_F$ equivalence classes as many times as needed to reach a fixed-point.  The first ordinal at which the fixed-point is reached is called the {\bf Cantor-Bendixson rank} of the linear order, denoted $CB(L; \leq)$.  
Observe that $CB(\Q; \leq) = 0$ and $CB(\omega; \leq) = 1$.  Moreover, the fixed-point reached after iteratively taking quotients of any linear order will either be isomorphic to the rational numbers or the linear order with a single element. A useful lemma tells us that automaticity is preserved as we take quotients by $\equiv_F$.

\begin{lemma}\label{lm:loQuotient}
If $\L = (L; \leq)$ is an automatic linear order then so is the quotient linear order $\L / \equiv_F$.  
\end{lemma}
\begin{proof}
By Proposition \ref{Prop:quotient}, it suffices to show that $\equiv_F$ is definable in the extended logic $(FO + \exists^\infty)$ of $(L; \leq)$.  The definition is 
\[
x \equiv_F y \iff \neg \exists^\infty z [( x \leq z ~\&~ z \leq y ) \vee ( y \leq z ~\&~ z \leq x ) ]
\] \qed
\end{proof}

Theorem \ref{thm:wfpo} showed that well-founded automatic partial orders have relatively low ordinal heights.  In this vein, it is reasonable to expect a low bound on the Cantor-Bendixson rank of automatic linear orders as well.  The following characterization theorem does just that.  This characterization theorem and its implications for linear orders and ordinals are discussed in \cite{KhRS03}, \cite{KhRS04}.

\begin{theorem}[Khoussainov, Rubin, Stephan; 2003]\label{thm:loCB}
An ordinal $\alpha$ is the Cantor-Bendixson rank of an automatic linear order if and only if it is finite.   
\end{theorem}

The proof of Theorem \ref{thm:loCB} has many common features with the proof of Theorem \ref{thm:wfpo}.  One direction is easy: for $n< \omega$, $\omega^n$ is automatic and $CB(\omega^n) =n$.  The hard direction relies on understanding the Cantor-Bendixson ranks of suborders of a given linear order.  In particular, we make use of suborders determined by intervals of a linear order: sets of the form $\{ z : x \leq z \leq y \}$ for some $x$ and $y$.

\begin{lemma}\label{lm:CBinterval}
For any linear order $\L$ and ordinal $\alpha$, if $CB(\L) = \alpha$ and $\beta \leq \alpha$ then there is an interval $[x,y]$ of $\L$ with $CB([x,y]) = \beta$.
\end{lemma}


To prove that any linear order with infinite Cantor-Bendixson rank is not automatic, we go by contradiction.  We suppose that such a linear order exists, and use the associated automata to define an equivalence relation on intervals which has only finitely many equivalence classes.  Moreover, intervals in the same equivalence class have the same Cantor-Bendixson rank.  However, since we assume that the linear order has infinite Cantor-Bendixson rank, Lemma \ref{lm:CBinterval} allows us to pick out intervals with every finite Cantor-Bendixson rank.  Therefore, two such intervals must be in the same equivalence class.  But this contradicts our choice of intervals with distinct Cantor-Bendixson ranks. \qed

\smallskip

Theorem \ref{thm:loCB} has been incredibly productive for decidability results.  In particular, it yields algorithms for computing the Cantor-Bendixson rank of a given automatic linear order, and for studying scattered linear orders and ordinals.  A linear order is called {\bf dense} if for each $x$ and $y$ with $x \leq y$, there is some $z$ such that $x \leq z \leq y$.  The linear orders with zero elements or one element are trivially dense.  For countable linear orders, there are exactly four isomorphism types of non-trivial dense linear orders (the rational numbers restricted to $(0,1)$, the rational numbers restricted to $[0,1)$, the rational numbers restricted to $(0,1]$,  the rational numbers restricted to $[0,1]$).  Note that being dense is a first-order definable property, so Corollary \ref{cr:KNdecidability2} tells us we can decide if a given automatic linear order is dense.  A linear order is called {\bf scattered} if it contains no non-trivial dense sub-order.   A linear order is an ordinal if it is well-founded.

\begin{corollary}
There is an algorithm which, given an automatic linear order $\L$, computes the Cantor-Bendixson rank of $\L$.
\end{corollary}
\begin{proof}
Check if $\L$ is dense.  If it is, output $CB(\L) = 0$.  Otherwise, Lemma \ref{lm:loQuotient} tells us that the quotient $\L / \equiv_F$ is automatic.  We iterate checking if the quotient is dense and, if it is not, constructing the next quotient and incrementing the counter.  Each of these steps is effective because denseness is a first-order question.  Moreover, this procedure eventually stops by Theorem \ref{thm:loCB}.  Once we reach a dense quotient structure, we output the value of the counter. \qed
\end{proof}

The following two corollaries about automatic scattered linear orders use trivial modifications of the above algorithm.

\begin{corollary}
It is decidable if a given automatic linear order is scattered. 
\end{corollary}

\begin{corollary}
Given an automatic linear order $\L$ that is not scattered, there is an algorithm which computes an automatic dense suborder of $L$.
\end{corollary}

We now apply Theorem \ref{thm:loCB} to the subclass of linear orders which are well-founded, the ordinals.  We will see that we can effectively check if a given automatic linear order is an ordinal; and given two automatic ordinals, we can check if they are isomorphic.  The isomorphism question is one of the central motivating questions in the study of automatic structures (see Subsection \ref{subsec:Questions}).  The class of automatic ordinals was one of the first contexts in which a positive answer to this question was found. 

\begin{corollary}
If $\L$ is an automatic linear order, there is an algorithm which checks if $\L$ is an ordinal.
\end{corollary}

\begin{proof}
To check if a given automatic linear order $\L$ is an ordinal, we need to check if it has an infinite descending sequence.   Note that infinite descending sequences can occur either within an $\equiv_F$ equivalence class or across such classes.  We begin by checking whether $\L$ is not dense and $\forall (x \in L) \exists^\infty y (x \equiv_F y ~\&~ y < x )$. If this condition holds, we form the quotient  $\L /\equiv_F$  and check the condition again for the quotient linear order.  We iterate until the condition fails, which must occur after finitely many iterations because $CB(\L)$ is finite.   If the resulting linear order has exactly one element, output that $\L$ is an ordinal, and otherwise output that it is not. If $\L$ is an ordinal then the $\equiv_F$ equivalence classes are all finite or isomorphic to $\omega$ and all quotient linear orders of $\L$ are also ordinals.  Therefore, the algorithm will stop exactly when the quotient is dense, in which case it will have exactly one element.  If $\L$ is not an ordinal, then either there will be a stage of the algorithm at which there is an infinite descending chain within a single $\equiv_F$ equivalence class or the final dense linear order will contain an infinite descending chain.  In either case, the algorithm will recognize that $\L$ is not an ordinal. \qed
\end{proof}

\begin{corollary}
The Cantor normal form of a given automatic ordinal is computable. 
\end{corollary}
\begin{proof}
Given an automatic ordinal $\L$, we use first-order definitions of maximal elements and the set of limit ordinals in $\L$ to iteratively determine the coefficients in the Cantor normal form.  The set of limit ordinals play a role because if $\alpha = \omega^m a_m + \cdots + \omega a_1$, then the Cantor normal form of the set of limit ordinals strictly below $\alpha$ is $\omega^{m-1} a_m + \cdots + \omega^1 a_2 + a_1$. \qed
\end{proof}

Since two ordinals are isomorphic if and only if they have the same Cantor normal form, the following corollary is immediate.

\begin{corollary}
The isomorphism problem for automatic ordinals is decidable. 
\end{corollary}

We do not know whether the isomorphism problem for automatic linear orders is decidable.

\subsection{Word automatic trees}\label{sec:trees}

 We now focus our interest on partial orders which form trees.  We begin with the following definition of trees.

\begin{definition}
A {\bf (partial order) tree} is $\A = (A; \leq)$ where $\leq$ is a partial order such that $A$ has a $\leq$-least element and the set $x \downarrow$ (the $\leq$-predecessors of $x$) is linearly ordered and finite for all $x \in A$.
\end{definition}

For any regular language $L$, if $L$ is prefix-closed then the structure $(L; \preceq)$ where $\preceq$ is the prefix relation is a (word) automatic tree.  The two most famous such examples are the full binary tree $(\{0,1\}^\star; \preceq)$ and the countably branching tree $\omega^{< \omega} \cong (\{0,1\}^\star \cdot 1; \preceq)$.  In the following, for a given tree $\T = (T; \leq)$ we denote by $\T_x$ the subtree of $\T$ which is rooted at node $x$; that is, $\T_x$ is the tree $(\{ y \in T : x \leq y \}; \leq)$.

\smallskip

There is a natural connection between trees and topological spaces: given a tree $\T = (T; \leq)$, define an associated topological space whose elements are infinite paths of $\T$ and whose basic open sets are collections of infinite paths defined by $\{ P : x \in P \}$ for each $x \in T$.  Then the topological Cantor-Bendixson rank transfers to trees in a straightforward way.  
As in the case of linear orders, we can compute this rank via an iterative search for a fixed-point.  The derivative  is defined to be $d(\T)$, the subtree of $\T$ containing all nodes $x \in T$ such that $x$ belongs to two distinct infinite paths of $\T$.  The derivative can be carried along the ordinals by letting $d^{\alpha+1}(\T) = d(d^\alpha (\T))$ and, for limit ordinals $\gamma$, $d^{\gamma}(\T) = \cap_{\beta < \gamma} d^{\beta}(\T)$.  The first ordinal $\alpha$ for which $d^{\alpha}(\T) = d^{\alpha+1}(\T)$ is called the Cantor-Bendixson rank of $\T$ and is denoted by $CB(\T)$.  The proof of the following lemma gives some intuition about Cantor-Bendixson ranks and is left to the reader.

\begin{lemma}\label{lm:CBsucc}
If a tree has only countably many paths, its Cantor-Bendixson rank is $0$ or a successor ordinal.
\end{lemma}

As in our discussion of linear orders, it is natural to look for low bounds on the complexity of automatic trees.  In particular, the Cantor-Bendixson ranks of trees and linear orders are intimately connected.  A linear order can be associated with each automatic finitely branching tree.  The {\bf Kleene-Brouwer ordering} of the tree $\T= (T; \leq)$ is given by $x \leq_{KB_\T} y$ if and only if $x =y$ or $y \leq x$ or there are $u,v,w \in T$ such that $v,w $ are immediate successors of $u$ and $v \leq_{llex} w$ and $v \leq x$ and $w \leq y$.  Note that the definition of the Kleene-Brouwer order uses the length-lexicographic ordering inherited by $T$ as a subset of $\Sigma^\star$ for some finite alphabet $\Sigma$. It is not hard to see that the Kleene-Brouwer order is a linear order.  For example, the Kleene-Brouwer order associated to the tree which has exactly two infinite paths $(0^\star \cup 1^\star; \preceq)$ is isomorphic to $\omega^\star + \omega^\star$.  If $\T$ is an automatic tree then the relation $\leq_{KB_\T}$ is first-order definable in it.  Hence, the results in Subsection \ref{sec:AutPo} imply that the Kleene-Brouwer ordering of an automatic finitely branching tree is automatic.

\begin{lemma}[Khoussainov, Rubin, Stephan; 2003]\label{lm:CBtreeLo}
If $\T$ is an automatic finitely branching tree with countably many paths then $CB(\T)= CB(KB_\T)$.
\end{lemma}
\begin{proof}
The outline of the proof is as follows; for details, see \cite{KhRS03}.  We claim that if $\T$ is an automatic finitely branching tree with countably many paths then $KB_\T$ is scattered. Given any infinite path of $\T$, the linear order $KB_\T$ can be expressed as a countable sum of linear orders which are either trivial or $KB_{\S}$ for $\S$ a subtree of $\T$.  By induction, we will show that $CB(\T) = CB(KB_\T)$.  If $CB(\T) = 0$ then it is empty so $KB_\T$ is also empty and $CB(KB_\T) = 0$.  For the inductive step, we suppose that $CB(\T) = \beta + 1$ and notice that the subtree of $\T$ whose domain is $X = \{ x \in T : CB( \T_x) = \beta +1\}$ has a finite and nonzero number of infinite paths.  Each of these paths give rise to a sum expression of $KB_{\T_x}$, where $CB(KB_{\T_x}) = \beta +1$.  Then $KB_\T$ is the sum of these finitely many linear orders $KB_{\T_x}$ and so $KB_{\T} = \beta+1$ as well.
\qed
\end{proof}

The lemma above allows us to transfer the bound on Cantor-Bendixson ranks of automatic linear orders to a bound on the Cantor-Bendixson ranks of automatic finitely branching trees with countably many paths.

\begin{theorem}[Khoussainov, Rubin, Stephan; 2003]\label{thm:CBfinbrcountpathpotree}
If $\T$ is an automatic finitely branching tree with countably many paths then $CB(\T)$ is finite.
\end{theorem}
\begin{proof}
Suppose $\T$ is as above.  Then $KB_\T$ is also automatic.   By Theorem \ref{thm:loCB}, $CB(KB_\T)$ is finite.  Lemma \ref{lm:CBtreeLo} then gives that  $CB(\T)$ is finite. \qed
\end{proof}

We will improve Theorem \ref{thm:CBfinbrcountpathpotree} by weakening its hypotheses.  Our goal is to show that the Cantor-Bendixson rank of each automatic partial order tree is finite.

\begin{theorem}[Khoussainov, Rubin, Stephan; 2003]\label{thm:CBfinbranchpotree}
If $\T$ is an automatic finitely branching tree then $CB(\T)$ is finite.
\end{theorem}
\begin{proof}
We consider two kinds of nodes on $\T$: those which lie on at most countably many infinite paths, and those which lie on uncountably many infinite paths.  For each $x$ which lies on at most countably many infinite paths consider $\T_x$, the subtree of $\T$ rooted at $x$.  Note that $KB_{\T_x}$ is an interval of $KB_\T$ for each $x$.  By Lemma \ref{lm:CBtreeLo}  and Theorem \ref{thm:loCB}, $CB(\T_{x}) = CB(KB_{\T_x}) \leq CB(KB_\T) = n$ for some $n$.  Therefore, $d^n (\T)$ contains only nodes which lie on uncountably many infinite paths.  In particular, this implies that $d^{n+1}(\T) = d^n(T)$ so $CB(\T) = n$.
\qed
\end{proof}

Finally, we remove the requirement that the tree is finitely branching.

\begin{theorem}[Khoussainov, Rubin, Stephan; 2003]\label{thm:CBpotree}
If $\T$ is an automatic tree then $CB(\T)$ is finite.
\end{theorem}
\begin{proof}
Given an automatic tree $\T = (T; \leq)$, we consider a finitely branching tree which is first-order definable in $\T$ and whose Cantor-Bendixson rank is no less than that of $\T$.  The tree $\T'$ has domain $T$ and partial order $\leq'$ defined by 
\[
x \leq' y \iff x \leq y \vee \exists v \exists w [ x \in S(v) ~\&~ w \in S(v) ~\&~ x \leq_{llex} w ~\&~ w \leq y ),
\]
where $S$ is the immediate successor function with respect to $\leq$.  To picture $\T'$, note that each node has at most two immediate successors: its $\leq_{llex}$-least successor from $\T$, and its $\leq_{llex}$-least sibling in $\T$.  We can define a continuous and injective map from the infinite paths of $\T$ to the infinite paths of $\T'$.  Analysing the derivatives of $\T$ and $\T'$ via this map leads to the conclusion that $CB(\T) \leq CB(\T')$.
Applying Theorem \ref{thm:CBfinbranchpotree} to $\T'$, $CB(\T)$ is finite. \qed
\end{proof}

\begin{corollary}
An ordinal $\alpha$ is the Cantor-Bendixson rank of an automatic (partial order) tree if and only if it is finite.   
\end{corollary}

One of the most fundamental tools available for the study of finitely branching trees is K\"onig's lemma: every finitely branching infinite tree has an infinite path. It is natural to wonder whether such a result holds in the context of automatic structures.  There are several ways one can transfer results from mathematics to include feasibility constraints.  For example, we can ask whether every automatic finitely branching infinite tree has a regular infinite path.  A stronger result would be to find a regular relation which picks out the infinite regular paths through automatic trees.  We will now develop several automatic versions of K\"onig's lemma in this spirit.  As a starting point, recall Corollary \ref{cr:autKL1}: every infinite automatic finitely branching tree has a regular infinite path.
A surprising feature of the landscape of automatic structures  is that a version of K\"onig's lemma holds even when we remove the finitely branching assumption.

\begin{theorem}[Khoussainov, Rubin, Stephan; 2003]\label{thm:KL2}
If an infinite automatic tree has an infinite path, it has a regular infinite path.
\end{theorem}
\begin{proof}
The first-order definition of an infinite path in Corollary \ref{cr:autKL1} relied on a definition of the set of nodes which lie on any infinite path.  However, in the context of trees which may have nodes with infinite degree, such a definition is harder to find.  To show that this set of nodes is regular, we define an auxiliary B\"uchi recognizable set which tracks nodes on infinite paths.  Our desired regular set is then achieved via applications of the projection operation and decoding of the B\"uchi automaton.  For more details, see \cite{KhRS03}. \qed
\end{proof}

A full automatic version of K\"onig's lemma would produce a regular relation which codes  all infinite paths through a given automatic tree.  The following theorem  from \cite{KhRS03} does just this.

\begin{theorem}[Khoussainov, Rubin, Stephan; 2003]\label{thm:KL3}
Given an automatic tree $\T$, there is a regular relation $R(x,y,z)$ satisfying the following properties:
\begin{itemize}
\item $\exists y\exists z~ \big (R(x,y,z) \iff  \T_x ~\text{has at most countably many infinite paths} \big) $;
\item for each $y \in \Sigma^\star$ and for each $x$ such that $\T_x$ has at most countably many infinite paths, the set $R_{x,y} = \{ z : R(x,y,z) \}$ is either empty or is an infinite path through $\T_x$; and
\item for each $x$ such that $\T_x$ has at most countably many infinite paths, if $\alpha$ is an infinite path through $\T_x$ then there is a $y \in \Sigma^\star$ such that $R_{x,y} = \alpha$.
\end{itemize}
\end{theorem}

In this section and in the preceding one, we gave tight bounds on the ranks and heights of automatic well-founded partial orders, trees, ordinals, and partial order trees.  Such bounds indicate that the level of complexity of the structures in these classes is relatively low.  However, it does not give us concrete information on the members of these classes.  In the next section, 
we present classification theorems for various collections of automatic structures.

\subsection{Word automatic Boolean algebras}\label{sec:autBA}

We now use the results about non-automaticity from Subsection \ref{sec:NonAut} to give a classification of word automatic Boolean algebras.  Recall that a {\bf Boolean algebra} is a structure $(D; \vee, \wedge, \neg, 0, 1)$ which satisfies axioms relating the join ($\vee$), meet ($\wedge$), and complement ($\neg$) operations and the constants $0$ and $1$.  An {\bf atom} of a Boolean algebra is a minimal non-$0$ element.  An archetypal Boolean algebra is $\B_{\omega}$, the collection of all finite and co-finite subsets of $\omega$.  In this case, the Boolean operations are the Boolean set operations: $\vee = \cup$, $\wedge = \cap$, $\neg = ~^{c}$.   As we saw in Subsection \ref{sec:AutPresent}, $\B_{\omega}$ has a word automata presentation.  For each $n \geq 1$, we define $\B_{\omega}^n$ to be the $n$-fold Cartesian product of $\B_{\omega}$.  Since finite Cartesian products preserve automaticity, $\B_\omega^n$ is automatic.  The main theorem of this section says that each automatic Boolean algebra must be isomorphic to $\B_\omega^n$ for some $n$.

\begin{theorem}[Khoussainov, Nies, Rubin, Stephan; 2004]\label{thm:autBAchar}
A Boolean algebra is word automatic if and only if it is isomorphic to $\B^n_\omega$ for some $n \geq 1$.
\end{theorem} 

The discussion above shows that any Boolean algebra isomorphic to $\B_\omega^n$ is word automatic.  We now prove the converse.  Let $\B$ be an automatic Boolean algebra.  Assume for a contradiction that $\B$ is not isomorphic to any $\B_\omega^n$.  
We will construct a set of elements of $B$ which will contradict the Growth Lemma for Monoids (Lemma \ref{lm:GenMonoids}).  To do so, we recall some terminology for elements of Boolean algebras.  We say that $a,b \in B$ are $F$-equivalent, $a \equiv_F b$, if the element $(a \wedge \bar{b}) \vee (\bar{a} \wedge b)$ is a union of finitely many atoms of $\B$.  Note that $\B / \equiv_F$ is itself a Boolean algebra.  Moreover,  $\B/ \equiv_F$ is automatic because the equivalence relation $\equiv_F$ is $(FO + \exists^\infty)$ definable in $\B$ (Proposition \ref{Prop:quotient}).   Moreover, since $\B$ is not isomorphic to $\B_\omega^n$ for any $n$, $\B / \equiv_F$ is infinite.  We call $x \in B$ {\bf infinite} if there are infinitely many elements $y \in B$ such that $y \wedge x = y$ (i.e. $y \leq x$ in the partial order induced by the Boolean algebra).  We say that $x$ {\bf splits} $y$ in $\B$ if $x \wedge y \neq 0$ and $\bar{x} \wedge y \neq 0$.
We call $x \in B$ {\bf large} if the $\equiv_F$ equivalence class of $x$ is not a finite union of atoms of the quotient algebra.   The following lemma collects properties of infinite and large elements of a Boolean algebra.

\begin{lemma}\label{lm:BAlargeinfinite}
If $y$ is large then there is $x$ that splits $y$ such that $y \wedge x$ is large and $y \wedge \bar{x}$ is infinite.  If $y$ is infinite then there is $x$ that splits $y$ such that either $x \wedge y$ is infinite or $\bar{x} \wedge y$ is infinite.
\end{lemma}

We use Lemma \ref{lm:BAlargeinfinite} to inductively define a sequence of trees $\{\T_n\}$ whose nodes are binary strings.  To each node $\sigma$ of the tree $\T_n$ we associate an element $b_\sigma$ of the Boolean algebra.  We denote the set of elements of $\B$ associated to the leaves of the tree $\T_n$ by $X_n$.  At each step of the induction, we verify that the following properties hold:
\begin{itemize}
\item There is a leaf $\sigma\in \T_n$ such that $b_\sigma$ is large in $\B$.
\item There are $n$ leaves $\sigma_1, \ldots, \sigma_n \in \T_n$ such that $b_{\sigma_i}$ is infinite.
\item The tree $\T_n$ has as least $\frac{n (n+1)}{2}$ many leaves.
\item If $\sigma$ and $\tau$ are distinct leaves of $\T_n$, $b_\sigma \wedge b_\tau = 0$.
\end{itemize}
The base tree is $T_{0} = \{ \lambda\}$ (where $\lambda$ is the empty binary string) and $b_{\lambda} = 1$.  Given $\T_n$, we define $\T_{n+1}$ as the extension of $\T_n$ obtained by doing the following for each leaf $\sigma \in \T_n$.  If $b_\sigma$ is large, let $a \in B$ be the length lexicographically first element that splits $b_\sigma$ such that both $b_\sigma \wedge a$ and $b_\sigma \wedge \bar{a}$ are infinite and one of them is large.  Such an $a$ exists by Lemma \ref{lm:BAlargeinfinite}.  Add $\sigma 0$ and $\sigma 1$ to $\T_{n+1}$ (as leaves) and let $b_{\sigma 0} = b_\sigma \wedge a$ and $b_{\sigma 1} = b_\sigma \wedge \bar{a}$.   If $b_\sigma$ is one of the $n$ leaves from $\T_n$ for which $b_{\sigma}$ is infinite but is not large,  let $a \in B$ be the length lexicographically first element of $B$ that splits $b_{\sigma}$ such that one of $b_\sigma \wedge a$ and $b_\sigma \wedge \bar{a}$ is infinite.  Again, this $a$ exists by Lemma \ref{lm:BAlargeinfinite}.  Add $\sigma 0$ and $\sigma 1$ to $\T_{n+1}$ and let $b_{\sigma 0} = b_{\sigma} \wedge a$ and $b_{\sigma 1} = b_{\sigma} \wedge \bar{a}$.  For any leaf $\sigma \in \T_n$ that does not fall into either of these cases, let $\sigma \in \T_{n+1}$ also be a leaf.  

\smallskip

The first induction hypothesis holds for $\T_{n+1}$ because any leaf in $\T_n$ associated to a large element in $\B$ is extended by at least one node associated to a large element in $\B$.  Likewise, the second hypothesis is preserved because each of the $n$ nodes associated to infinite elements are extended by a node associated to an infinite element, plus the node associated to a large element is extended by a node associated to an infinite element.  The third hypothesis holds because at least $n+1$ leaves of $\T_n$ are replaced by two leaves each, so if $\T_n$ has at least $\frac{n (n+1)}{2}$ leaves then $\T_{n+1}$ has at least $\frac{(n+1) (n+2)}{2}$ many leaves.
Finally, the last hypothesis is preserved because if $\tau \in \T_n$ is a leaf, $\tau \wedge (b_\sigma \wedge a) = (\tau \wedge b_\sigma) \wedge a = 0 \wedge a = 0$. 

\smallskip

Each clause in the inductive definition can be formalized in $(FO + \exists^\infty)$.  Hence, for all $n$, $\T_n$ is word automatic.  In particular, the functions $f_0, f_1: X_n \to X_n$ which map $f_0(b_\sigma) = b_{\sigma 0}$ and $f_1(b_\sigma) = b_{\sigma 1}$ have regular graphs.  Therefore, the Constant Growth Lemma (Lemma \ref{lm:ConstantGrowth}) implies that there are constants $C_0, C_1$ such that
\[
|b_{\sigma 0}| \leq |b_\sigma| + C_0 ~\text{and}~|b_{\sigma 1}| \leq |b_\sigma| + C_1
\]
In particular, there is a constant $C_2$ such that for all $x \in X_n$, $|x| \leq C_2 \cdot n$.  In other words, $X_n \subseteq \Sigma^{C_2 \cdot n}$.  Lemma \ref{lm:GenMonoids} then gives that the Boolean algebra generated by $X_n$ is also a subset of $\Sigma^{O(n)}$ and therefore has at most $2^{O(n)}$ elements.  On the other hand, since there are at least $\frac{n (n+1)}{2}$ many leaves in $\T_n$ and distinct leaves yield disjoint elements of $X_n$, the Boolean algebra generated by $X_n$ has size at least $2^{\frac{n (n+1)}{2}}$.  This is a contradiction, and proves Theorem \ref{thm:autBAchar}. \qed

\smallskip

The classification of Boolean algebras in Theorem \ref{thm:autBAchar} gives rise to an algorithm solving the isomorphism problem for automatic Boolean algebras.

\begin{corollary}
It is decidable whether two automatic Boolean algebras are isomorphic.
\end{corollary}

\begin{proof}
By Theorem \ref{thm:autBAchar}, two automatic Boolean algebras are isomorphic if and only if they are isomorphic to $B_{\omega}^n$ for the same $n$.  Given an automata presentations of a Boolean algebra, it satisfies
the first-order expressible property \lqt there are $n$ disjoint elements each with infinitely many atoms below, and for each $m > n$ there aren't $m$ disjoint elements each with infinitely many atoms below them\rqt if and only if it is isomorphic to $B_\omega^n$.  To decide if two automatic Boolean algebras $\B_1, \B_2$ are isomorphic, search for $n_1$ and $n_2$ such that $\B_1 \cong \B_\omega^{n_1}$ and $\B_2 \cong \B_\omega^{n_2}$ and then reply \lqt yes\rqt if and only if $n_1 = n_2$. \qed
\end{proof}

We have seen explicit descriptions for which ordinals and which Boolean algebras have automata presentations.  We have also seen theorems which give conditions on invariants (ordinal height, Cantor-Bendixson rank) of structures with automatic presentations.  In the following subsection, we will continue to accumulate results of both kinds.

\subsection{Word automatic finitely generated groups}\label{sec:AutGroups}

We now turn to finitely generated groups which have automata presentations.  As in the previous subsections, we consider only word automata presentations.  Recall that a group is a structure $(G; \cdot, ~^{-1}, e)$ where the group operation is associative and the inverse behaves as expected with respect to the identity. In Subsection \ref{sec:AutPresent}, we saw that finitely generated abelian (commutative) groups have automata presentations.  In Subsection \ref{sec:NonAut}, we saw that the free group $F(n)$ with $n > 1$ is not automatic.  These two examples represent extreme behaviours with respect to automata.  We now work towards finding the exact boundary between automaticity and non-automaticity for groups. 

\smallskip

Recall that a {\bf subgroup} is a subset of a group which is itself a group; the {\bf index} of a subgroup is the number of left cosets of the subgroup.  A good introduction to basic group theory is \cite{Rotman}.

\begin{definition}
A group is {\bf virtually abelian} if it has an abelian subgroup of finite index.  Similarly, for any property $X$ of groups, we say that a group is {\bf virtually $X$} if it has a subgroup of finite index which has property $X$.  
\end{definition}

\begin{definition}
A group is {\bf torsion-free} if the only element of finite order (the least number of times it must be multiplied by itself to yield the identity) is the identity. A subgroup $N$ of $G$ is {\bf normal} if for all $a \in G$, $aN = Na$.
\end{definition}

The following basic fact from group theory will play a key part in our analysis of automata presentable finitely generated groups.  It says that the abelian subgroup of a virtually abelian finitely generated group can be assumed to have a special form.

\begin{fact}\label{ft:GroupTheory}
Every virtually abelian finitely generated group has a torsion-free normal abelian subgroup of finite index.
\end{fact}

We are now ready to prove that each member of a large class of finitely generated groups has an automata presentation.  This lemma significantly extends Example \ref{ex:abfg} which dealt with abelian finitely generated groups.

\begin{lemma}
Any virtually abelian finitely generated group is automatic.
\end{lemma}
\begin{proof}
Let $G$ be a virtually abelian finitely generated group.  By Fact \ref{ft:GroupTheory}, let $N = \la x_1, \ldots, x_k \ra$ be an abelian torsion-free normal subgroup of finite index of $G$.  Thus, there are $a_1, \ldots, a_n \in G$ such that $G = \sqcup_{i=1}^n a_i N$.  Since $N$ is normal, for any $j = 1,\ldots, n$ there are $h_1, \ldots, h_k \in N$ such that 
\[
x_1 a_j = a_j h_1 = a_j x_1^{m_{1,1}(j)} \cdots x_k^{m_{k,1}(j)}, ~\ldots~,
x_k a_j = a_j h_k = a_j x_1^{m_{1,k}(j)} \cdots x_k^{m_{k,k}(j)}.
\]
Moreover, there is a function $f: \{1, \ldots, n \} \times \{1, \ldots, n \} \to  \{1, \ldots, n \}$ such that
\[
a_i a_j  = a_{f(i,j)} h_{(i,j)} = a_{f(i,j)} x_1^{m_{1}(i,j)} \cdots x_k^{m_k(i,j)}
\]
For each $g \in G$ there is $h \in N$ and $i \in 1, \ldots, n$ so that $g = a_i h = a_i x_1^{m_1} \cdots x_k^{m_k}$.  Therefore, we can express the product of two elements of the group as
\begin{align*}
(a_i x_1^{p_1} \cdots x_k^{p_k} ) \cdot (a_j x_1^{q_1} \cdots x_k^{q_k} ) = &a_{f(i,j)}x_1^{q_1 + m_1(i,j)+\Sigma_{\ell=1}^k p_\ell m_{1,\ell}(j)  } \cdots \\
&\cdot x_k^{q_k + m_k(i,j) + \Sigma_{\ell=1}^k p_\ell m_{k,\ell}(j) }
\end{align*}
Expressing group multiplication in this way is well-suited to automata operations and leads to an automata presentation of $G$. \qed
\end{proof}

In fact, \cite{OliverThomas} contains the following theorem which shows that the class of virtually abelian finitely generated groups coincides exactly with the class of automata presentable finitely generated groups.

\begin{theorem}[Oliver, Thomas; 2005]\label{thm:autFGgroups}
A finitely generated group is automatic if and only if the group is virtually abelian.
\end{theorem}

To prove Theorem \ref{thm:autFGgroups}, it remains to prove only one direction of the classification.  We will use several definitions and facts from group theory to prove this direction (cf. \cite{Gromov}, \cite{Romanovskii}, \cite{Noskov}).

\begin{definition}
The {\bf commutator} of a group $G$ is the set 
\[
[G,G] = \{ [g,h] = g^{-1} h^{-1} gh : g, h \in G \}.  
\]
The powers of a group are defined inductively as
$G^0 = G$ and $G^{k+1} = [G_k,G_k]$.
The group $G$ is {\bf solvable} if there is some $n$ such that $G^n = \{e\}$.  The maps $\gamma_k$ are defined inductively by
\[
\gamma_0 (G) = G, \qquad \gamma_{k+1} (G) = [\gamma(G_k), G].
\]
The group $G$ is {\bf nilpotent} if $\gamma_n(G) =\{e \}$ for some $n$.
\end{definition}
\begin{fact}\label{fct:Groups}
If $G$ is nilpotent then $G$ is solvable.
\end{fact}

The following theorems of group theory relate algorithmic and growth properties of a group to the group theoretic notions defined above.  These theorems are instrumental in proving Theorem \ref{thm:autFGgroups} since we understand the algorithmic and growth properties of groups with automata presentations.  Note that a finitely generated group is said to have {\bf polynomial growth} if the size of the $n^{th}$ generation ($G_n$) of the set of generators of the group is polynomial in $n$ (recall the definition of the generations of a structure from Subsection \ref{sec:NonAut}).

\begin{theorem}[Gromov; 1981]
If a finitely generated group has polynomial growth then it is virtually nilpotent.
\end{theorem}


\begin{theorem}[Romanovski$\bf{\breve{\i}}$; 1980. Noskov; 1984]
A virtually solvable group has a decidable first-order theory if and only if it is virtually abelian.
\end{theorem}

\begin{proof}[Theorem \ref{thm:autFGgroups}]
Let $G$ be an automatic finitely generated group. Suppose $G = \la a_1, \ldots, a_k \ra$.  By the generation lemma for monoids (Lemma \ref{lm:GenMonoids}), for each $n$, $G_n( \{ a_1, \ldots, a_k \}) \subseteq \Sigma^{C \cdot \log(n) }$.  Therefore, $|G_n(\{a_1, \ldots, a_k \})| \leq n^C$, and so $G$ has polynomial growth.  By Gromov's theorem, $G$ is virtually nilpotent, hence Fact \ref{fct:Groups} implies that it is virtually solvable.  Since $G$ is automatic, it has a decidable first-order theory.  Hence, Romanovski$\breve{\i}$ and Noskov's theorem implies that $G$ is virtually abelian, as required. \qed
\end{proof}

We have just seen a complete description of the finitely generated groups which have word automata presentations.  In the Boolean algebra case, such a description led to an algorithm for the isomorphism question.  However, whether such an algorithm exists in the current context is still an open question: is the isomorphism problem for automatic finitely generated groups decidable?

\smallskip

Many other questions can be asked about automata presentations for groups.  For example, we might shift our attention away from finitely generated groups and ask whether the isomorphism problem for automatic torsion-free abelian groups is decidable. A problem which has been attempted without success by many of the most distinguished researchers in automatic structures is to determine whether the additive group of the rational numbers $(\Q, +)$ has a word automata presentation.  More details about the current state of the art in automata presentable groups may be found in the survey paper \cite{NiesSurvey}.

\section{
\em Complicated Structures}
\subsection{Scott ranks of word automatic structures}\label{sec:Scott}

In the previous lecture we saw positive characterization results and relatively low tight bounds on classes of automatic structures.  In particular, we saw that $\omega$ is the tight bound on the Cantor-Bendixson ranks of automatic partial order trees, and that $\omega^\omega$ is the tight bound on the ordinal heights of automatic well-founded relations.  We will now prove results at the opposite end of the spectrum: results that say that automatic structures can have arbitrarily high complexity in some sense.  This will have complexity theoretic implications on the isomorphism problem for the class of automatic structures.

\smallskip

We begin by recalling an additional notion of rank for the class of countable structures.  This {\bf Scott rank} was introduced in \cite{Sco65} in the context of Scott's famous isomorphism theorem for countable structures.
\begin{definition}
Given a countable structure $\A$, for tuples $\bar{a}, \bar{b} \in A^n$ we define the following equivalence relations.
\begin{itemize}
\item $\bar{a} \equiv^0 \bar{b}$ if $\bar{a}$ and $\bar{b}$ satisfy the same quantifier free sentences,
\item for ordinal $\alpha > 0$, $\bar{a} \equiv^\alpha \bar{b}$ if for all $\beta < \alpha$, for each tuple $\bar{c}$ there is a tuple $\bar{d}$ such that $\bar{a}, \bar{c} \equiv^\beta \bar{b}, \bar{d}$ and for each tuple $\bar{d}$ there is a tuple $\bar{c}$ such that $\bar{a}, \bar{c} \equiv^\beta \bar{b}, \bar{d}$.
\end{itemize}
The Scott rank of tuple $\bar{a}$ is the least $\beta$ such that for all $\bar{b} \in A^n$, $\bar{a} \equiv^\beta \bar{b}$ implies $(\A, \bar{a}) \cong (\A, \bar{b})$.  The {\bf Scott rank} of the structure $\A$, denoted $SR(\A)$, is the least ordinal greater than the Scott ranks of all tuples of $\A$.
\end{definition}

The Scott rank was extensively studied in the context of computable model theory.  In particular, Nadel \cite{Nadel85} and Harrison \cite{Harr68} showed that the tight upper bound for the Scott rank of computable structures is $\omega_1^{CK}+1$.  Recall that $\omega_1^{CK}$ is the first non-computable ordinal; that is, it is the least ordinal which is not isomorphic to a computable well-ordering of the natural numbers.  However, most common examples of automatic structures have low Scott ranks.  The following theorem from \cite{KM} proves that automatic structures have the same tight upper bound on the Scott rank as computable structures.  

\begin{theorem}[Khoussainov, Minnes; 2007]\label{thm:ScottRanks}
For each $\alpha \leq \omega_1^{CK} + 1$, there is an automatic structure with Scott rank at least $\alpha$.  Moreover, there is an automatic structure with Scott rank $\omega_1^{CK}$.
\end{theorem}
\begin{proof}
We outline the idea of the proof.  The main thrust of the argument is the transformation of a given computable structure to an automatic structure which has similar Scott rank.  Let $\C = (C; R_1, \ldots, R_m)$ be a computable structure.  For simplicity in this proof sketch, we assume $m=1$ and that $C = \Sigma^\star$ for some finite $\Sigma$.  Let $\M$ be a Turing machine computing the relation $R$.  The configuration space $Conf(\M)$ of the machine $\M$ is a graph whose nodes encode configurations of $\M$ and where there is an edge between two nodes if $\M$ has an instruction which takes it in one step from the configuration represented by one node to that of the other (see Example \ref{ex:conf}).   Recall that $Conf(\M)$ is an automatic structure.  We call a deterministic Turing machine {\bf reversible} if its configuration space is well-founded; that is, it consists only of finite chains or chains of type $\omega$.  The following lemma from \cite{Ben73} allows us to restrict our attention to reversible Turing machines.

\begin{lemma}[Bennett; 1973]
Any deterministic Turing machine may be simulated by a reversible Turing machine.
\end{lemma}

Without loss of generality, we assume that $\M$ is a reversible Turing machine which halts if and only if its output is \lqt yes\rqt.  We classify the chains in $Conf(\M)$ into three types:  terminating computation chains are finite chains whose base is a valid initial configuration, non-terminating computation chains are infinite chains whose base is a valid initial configuration, and unproductive chains are chains whose base is not a valid initial configuration. We perform the following smoothing operations to $Conf(\M)$ in order to capture the isomorphism type of $\C$ within the new automatic structure. Note that each of these smoothing steps preserves automaticity.  First, we add infinitely many copies of $\omega$ chains and finite chains of every finite size.  Also, we connect to each base of a computation chain a structure which consists of infinitely many chains of each finite length.  Finally, we connect representations of each tuple $\bar{x}$ in $C$ to the initial configuration of $\M$ given $\bar{x}$ as input.  We call the resulting automatic graph $\A$.  The following lemma can be proved using the defining equivalence relations of Scott rank and reflects the idea that the automatic graph contains witnesses to many properties of the computable structure.

\begin{lemma}\label{lm:ScottRank}
$SR(\C) \leq SR(\A) \leq 2 + SR(\C)$.
\end{lemma}

At this point, we have a tight connection between Scott ranks of automatic and computable structures.  The following lemmas from  \cite{GonKn02} and \cite{KnM} describe the Scott ranks that are realized by computable structures,

\begin{lemma}[Goncharov, Knight; 2002]\label{lm:SRcomp1}
For each computable ordinal $\alpha < \omega_1^{CK}$, there is a computable structure whose Scott rank  is above $\alpha$.
\end{lemma}

\begin{lemma}[Knight, Millar; in print]\label{lm:SRcomp2}
There is a computable structure with Scott rank $\omega_1^{CK}$.
\end{lemma} 

We apply Lemma \ref{lm:ScottRank} to Lemmas \ref{lm:SRcomp1} and \ref{lm:SRcomp2} to produce automatic structures with Scott ranks above every computable ordinal and at $\omega_1^{CK}$.  To produce an automatic structure with Scott rank $\omega_1^{CK}+1$, we apply Lemma \ref{lm:ScottRank} to Harrison's ordering from \cite{Harr68}.  This concludes the proof of Theorem \ref{thm:ScottRanks}.
\qed
\end{proof}

\begin{corollary}[Khoussainov, Rubin, Nies, Stephan; 2004]
The isomorphism problem for automatic structures is $\Sigma_1^1$-complete.
\end{corollary}
\begin{proof}
In the proof of Theorem \ref{thm:ScottRanks}, the transformation of computable structures to automatic structures preserves isomorphism types.  Hence, the isomorphism problem for computable structures is reduced to that for automatic structures.  Since the isomorphism problem for computable structures is $\Sigma_1^1$-complete, the isomorphism problem for automatic structures is $\Sigma_1^1$-complete as well. \qed
\end{proof}

\subsection{More high bounds for word automatic structures}\label{sec:High}
The previous subsection contained an example where the behaviour of automatic structures matched that exhibited by the class of computable structures.  The techniques of Theorem \ref{thm:ScottRanks} can be used to obtain similar results in a couple of other contexts.  We first revisit the idea of automatic trees, introduced in Subsection \ref{sec:trees}.  In that subsection, we saw that all automatic partial order trees have finite Cantor-Bendixson rank (Theorem \ref{thm:CBpotree}).  Consider now a different viewpoint of trees.

\begin{definition}
A {\bf successor tree} is $\A = (A; S)$ where, if $\leq_{S}$ is the transitive closure of $S$, $(A; \leq_S)$ is a partial order tree.
\end{definition}

Observe that the successor tree associated with a given partial order tree is first-order definable in the partial order tree.  Hence, Corollary \ref{cr:Definability} implies that any automatic partial order tree is an automatic successor tree.  However, the inclusion is strict: the following theorem from \cite{KM} shows that the class of automatic successor trees is far richer than the class of automatic partial order trees. 

\begin{theorem}[Khoussainov, Minnes; 2007]\label{thm:CBsucc}
For each computable ordinal $\alpha < \omega_1^{CK}$ there is an automatic successor tree of Cantor-Bendixson rank $\alpha$.
\end{theorem}

The proof of Theorem \ref{thm:CBsucc} utilizes the configuration spaces of Turing machines, as in the proof of Theorem \ref{thm:ScottRanks}.  Another setting in which these tools are useful is that of well-founded relations (see \cite{KM}).  In Subsection \ref{sec:AutPo}, we proved that automatic well-founded partial orders have ordinal heights below $\omega^\omega$.  If we relax the requirement that the relation be a partial order, we can attain much higher ordinal heights.

\begin{theorem}[Khoussainov, Minnes; 2007]\label{thm:WFheights}
For each computable ordinal $\alpha < \omega_1^{CK}$ there is an automatic well-founded relation whose ordinal height is at least $\alpha$.
\end{theorem}

Theorem \ref{thm:WFheights} answers a question posed by Moshe Vardi in the context of program termination.  Given a program $P$, we say that the program is terminating if every computation of $P$ from an initial state is finite. If there is a computation from a state $x$ to $y$ then we say that $y$ is reachable from $x$. Thus, the program is terminating if the collection of all states reachable from the initial state is a well-founded set.  The connection between well-foundedness and program termination is explored further in \cite{BG06}.

\subsection{Borel structures}

Most of Lectures 2 and 3 so far dealt exclusively with word automatic structures.  To conclude this lecture, we turn again to automata on infinite strings (B\"uchi automata) and infinite trees (Rabin automata).  Recall from Subsection \ref{sec:TreeAut} that word automatic structures are a strict subset of tree automatic structures.  We would like a similar separation between B\"uchi automatic structures and Rabin automatic structures.  To arrive at such a separation, we recall a complexity hierarchy of sets from descriptive set theory (a good reference is \cite{Kech}).

\begin{definition}
A set is called {\bf Borel} in a given topology if it is a member of the smallest class of sets which contains all open sets and closed sets and is closed under countable unions and countable intersections.
\end{definition}

In the context of automatic structures, we have an underlying topology which depends on the objects being processed by the automata.  If the input objects are infinite binary strings (as in the case of B\"uchi automata), the basic open sets of the topology are defined as
\[
[\sigma] = \{ \alpha \in \{0,1\}^\omega : \sigma \prec \alpha \}
\]
for each finite string $\sigma$.  Based on this topology, we have the following definition.

\begin{definition}
A structure is {\bf Borel} if its domain and basic relations are Borel sets.
\end{definition}

\begin{fact}\label{ft:BucBorel}
Every B\"uchi automatic structure is Borel.
\end{fact}

We can use Fact \ref{ft:BucBorel} to prove that not all Rabin automatic structures are recognizable by B\"uchi automata.  \cite{HKMN08} contains the following theorem about Rabin automatic structures.

\begin{theorem}[Hjorth, Khoussainov, Montalb\'an, Nies; 2008]
There is a Rabin automatic structure that is not Borel and, therefore, is not B\"uchi automatic.
\end{theorem}

\begin{proof}
Consider the tree language $$V = \{ (\T, v) : \text{each path through} ~\T~ \text{has finitely many}~1s\}.$$  It is easy to see that $V$ is Rabin recognizable.  However, it is not Borel: consider the embedding of $\omega^{< \omega}$ into $\T$ which takes $n_1 \cdots n_k$ to $1^{n_1} 0 1^{n_2} 0 \ldots 1^{n_k}$.  The pre-image of $V$ is the set of trees in $\omega^{<\omega}$ which have no infinite path.  This set is an archetypal example of a non-Borel set, and hence $V$ is not Borel.  We will now transfer this example to the setting of structures by coding $V$ into a Rabin automatic structure.  The domain of this structure is the class of $\{0,1\}$-labelled trees $(\T, v)$.  The structure has two unary predicates: $S = \{ (\T, v) : \text{there is a unique}~x~\text{for which}~v(x) = 1\}$, and $V$ as above.  In addition, the structure contains two unary functions $Left', Right'$ which mimic the operations on $\T$.  Then $(D; S, V, Left', Right')$ is Rabin automatic.  But, if it had a Borel copy then $V$ would be Borel, and so the structure is not Borel. Thus, we have a Rabin automatic structure that is not B\"uchi automatic.\qed
\end{proof}

We conclude these lectures by proving Theorem \ref{thm:BucEquiv} from Subsection \ref{subsec:AutStructs}.  To do so, we need to exhibit a B\"uchi automatic structure and a B\"uchi recognizable equivalence relation such that their quotient is not B\"uchi recognizable. By Fact \ref{ft:BucBorel}, it suffices to show that the quotient is not Borel.  We will use the following well-known fact from descriptive set theory (see \cite{Hjorth} for a survey of relevant results in this area).  Recall that for $X, Y \subseteq \N$, $X =^\star Y$ means that $X$ and $Y$ agree except at finitely many points.

\begin{fact}\label{ft:BorelStar}
There is no Borel function $F: \P(\N) \to \{0,1\}^\omega$ such that for all $X, Y \subseteq \N$, $X =^\star Y$ if and only if $F(X) = F(Y)$
\end{fact}

\begin{proof}[Theorem \ref{thm:BucEquiv}]
The structure we will consider is an expansion of the disjoint union of two familiar structures.  Let $\B = ( \P(\N); \subseteq)$ and $\B^\star = (\P(\N)/ =^\star; \leq)$.  We will study the disjoint of union of $\B$ and $\B^\star$ along with a unary relation $U(x)$ satisfying $U(x) \iff x \in \P(\N)$ and a binary relation $R(x,y)$ interpreted as the canonical projection of $\P(\N)$ into $\P(\N) / =^\star$.

\smallskip

Suppose that $\B_0 = (B_0; \leq_0)$ and $\B_1=(B_1; \leq_1)$ are two disjoint presentations of $\B = (\P(\N); \subseteq)$.  Let $U^0$ pick out $B_0$ and $R^0$ be the bijection from $B_0$ to $B_1$ which acts like the identity on $\P(\N)$.  Then $\A = (B_0 \sqcup B_1; \leq_0 \sqcup \leq_1, U^0, R^0)$ is a B\"uchi  recognizable structure.  We define the equivalence relation
\[
E(x,y) \iff \begin{cases} x,y \in B_0 ~\&~ x=y \\
x,y \in B_1 ~\&~ x=^\star y \\
\end{cases}.
\]
It is not hard to see that $E(x,y)$ is B\"uchi recognizable given that $\B_0$ and $\B_1$ are. Observe that $\A / E \cong (\P(\N) \sqcup \P(\N)/ =^\star; \leq, U, R)$.  

\smallskip

Assume for a contradiction that $\A / E$ has a B\"uchi automata presentation.  By Fact \ref{ft:BucBorel}, this implies that $\A/ E$ has a Borel presentation $\C = (C; \leq^C, U^C, R^C)$.  Let $\Phi: \A/E \to \C$ be an isomorphism and $G: \B_0 \to \C$ be the restriction of $\Phi$ to $B_0$.  The following lemma tells us that $G$ is a Borel function.

\begin{lemma}
If $\S = ( S; \leq)$, $\S' = ( S'; \leq')$ are B\"uchi structures and $E, E'$ are B\"uchi equivalence relations such that $\S / E, \S'/ E' \cong (\P(\N); \subseteq)$ then if $\Psi: S/ E \to S'/ E'$ is an isomorphism, 
$graph(\Psi) = \{ \la x,y \ra \in S \times S' : \Phi( [x]_E ) = [y]_{E'} \}$
is Borel.
\end{lemma}

To prove the lemma, it suffices to show that $graph(\Psi)$ is a countable intersection of Borel sets.  Recall that $(\P(\N); \subseteq)$ can be viewed as a Boolean algebra whose atoms are singleton sets $\{n\}$ for $n \in \N$.  Let $\{ [a_n]_E : n \in \N \}$ be a list of atoms of $\S / E$ and let $b_n \in S'$ be such that $\Psi([a_n]_E ) = [b_n]_{E'}$.  
Since $\Psi$ is an isomorphism of Boolean algebras, for any $x \in S, y \in S'$
\[
\Psi([x]_E) = [y]_{E'} \iff \forall n (a_n \leq x \iff b_n \leq' y).
\]
Because $\S, \S'$ are B\"uchi structures, the relations $a_n \leq x$ and $b_n \leq' y$ are Borel and we have a Borel definition for $\Psi([x]_E) = [y]_{E'}$. The lemma is now proved and we return to the proof of Theorem \ref{thm:BucEquiv}.

\smallskip

The function $G: \B_0 \to \C$ satisfies the hypotheses of the above lemma and hence is Borel.  We define the map $F = R' \circ G$ from $\B_0 \cong \P(\N)$ to $\C$.  Note that since $F$ is the composition of Borels, it is Borel. Moreover, for any $X, Y \subseteq \N$
\begin{align*}
X =^\star Y &\iff R^0 (X) = R^0 (Y) \iff \Phi( R^0 (X)) = \Phi( R^0(Y))\\
 &\iff R' (G(X)) = R' (G(Y)) \iff F(X) = F(Y).
\end{align*}
contradicting Fact \ref{ft:BorelStar}.
\qed

\end{proof}


\begin{thebibliography}{16}
\bibliographystyle{alpha}

\bibitem{abl} Aceto, L., A. Burgueno, and K.G. Larsen. 
Model Checking via Reachability Testing for Timed Automata. 
\emph{Proc. TACAS '98}, 263-280, 1998.
 
\bibitem{akj} Abdulla, P.A.,  K. {\v C}er{\=a}ns, 
B. Jonsson , and Y. Tsay. Algorithmic Analysis of Programs with Well Quasi-Ordered
Domains. \emph{Proc. ICom}, {\bf 160}, 109-127, 2000.


\bibitem{bem}  Bouajjani, A., J. Esparza, and O. Maler.
Reachability Analysis of Pushdown Automata:  Application to Model Checking. \emph{Proc. CONCUR '97}, LNCS {\bf 1243}, 1997.

\bibitem{BarPhD} B\'ar\'any, V. \emph{Automatic Presentations of Infinite Structures}, Diploma Thesis, RWTH Aachen, 2007.

\bibitem{BKR}B\'ar\'any, V.,  L. Kaiser, and S. Rubin. Cardinality and counting quantifiers on omega-automatic structures. \emph{Proc. STACS '08}, to appear.


\bibitem{Ben73}Bennett, C.H. Logical Reversibility of Computation.  {\em IBM J. Res. Develop}, 525-532, 1973.

\bibitem{BG06}Blaas, A. and Y. Gurevich.  Program Termination and Well Partial Orderings.  {\em ACM Trans. Comput. Logic}, 1-25, 2006.

\bibitem{BluPhD} Blumensath, A.,  \emph{Automatic Structures}, 
        Diploma Thesis, RWTH Aachen, 1999.

\bibitem{BluGr00}Blumensath, A. and E. Gr\"adel.  Automatic Structures.  {\em Proc. LICS '00}, 51-62, 2000.

\bibitem{Buc60}B\"uchi, J.R.  Weak second-order arithmetic and finite automata. {\em Zeitschrift Math. Logik und Grundlagen det Mathematik}, 66-92, 1960.

\bibitem{Buc62}B\"{u}chi, J.R.  On a decision method in restricted second-order arithmetic. {\em Proc. Intern. Congress on Logic, Methodology and Philosophy of Science, 1960} (E. Nagel, P. Suppes, A. Tarski, Eds.), 1-12, 1962.


\bibitem{CeRe91}Cenzer, D.  and J. Remmel.  Polynomial-time versus recursive models.  {\em Annals of Pure and Applied Logic}, {\bf 54}, 17-58, 1991.  

\bibitem{Del04} Delhomm\'e, C. Automaticit\'e des ordinaux et des graphes homog\`enes.  {\em C.R. Acad\`emie des sciences Paris, Ser. I}, {\bf 339}, 5-10, 2004.

\bibitem{Doner65} Doner, J.E.  Decidability of the weak second-order theory of two successors. {\em Notices Amer. Math. Soc.}, {\bf 12}, 819, 1965.


\bibitem{EpsThurs} Epstein, D.B.A., et al.  {\em Word Processing in Groups}, A.K. Peters Ltd., 1992.

\bibitem{Ers68}Ershov, Yu.L.  Numbered Fields. {\em 1968 Logic, Methodology, and Philos. Sci. III (Proc. 3rd Internat. Congr. '67)}, 31-34, 1968.

\bibitem{Ers64}Ershov, Yu.L. Decidability of the elementary theory of distributive lattices with relative complements and the theory of filters, {\em Algebra i Logika} {\bf 3}, 17-38, 1964.

\bibitem{FS56}Fr\"olich, A. and J.C. Shepherdson. Effective Procedures in Field Theory, {\em Phil. Trans. Roy. Soc. London ser. A}, {\bf 248}, 407-432, 1956.

\bibitem{GonKn02}Goncharov, S.S. and J.F. Knight.  Computable structure and non-structure theorems.  {\em Algebra and Logic} {\bf 41}, 351-373, 2002.

\bibitem{Gromov} Gromov, M. Groups of polynomial growth and expanding maps, {\em Publications Math\'ematique dÕIH\'ES}, {\bf 53}, 53Ð78, 1981.

\bibitem{H1} Ershov, Yu.L, S.S. Goncharov, A. Nerode, J.B. Remmel, V.W. Marek (eds).  {\em Handbook of recursive mathematics Vol. 1}
Studies in Logic and the Foundations of Mathematics, {\bf 138}, North-Holland, 1998.
              
\bibitem{H2} Ershov, Yu.L., S.S. Goncharov, A. Nerode, J.B. Remmel, V.W. Marek (eds).  {\em Handbook of recursive mathematics Vol. 2}
Studies in Logic and the Foundations of Mathematics, {\bf 138}, North-Holland, 1998.

\bibitem{Hariz98}Harizanov, V.S. Pure Computable Model Theory.  {\em Handbook of Recursive Mathematics} (Yu.L. Ershov, S. Goncharov, A. Nerode, J. Remmel, eds.), 3-114,  1998.

\bibitem{Harr68}Harrison, J. Recursive Pseudo Well-Orderings.  {\em Trans. AMS}, {\bf 131: 2}, 526-543, 1968.

\bibitem{HKMN08}Hjorth, G., B. Khoussainov, A. Montalb\'an, and A. Nies  From Automatic Structures to Borel Structures.  {\em Proc. LICS '08}, to appear.

\bibitem{Hjorth}Hjorth, G. Borel equivalence relations.  {\em Handbook of Set Theory} (M. Foreman, A. Kanamori, M. Magidor, eds.), to apear.

\bibitem{H82}Hodgson, B.R.  On Direct Products of Automaton Decidable Theories.  {\em Theoret. Comp. Sci.}, {\bf 19}, 331-335, North-Holland, 1982.

\bibitem{Kech}Kechris, A.S. {\em Classical Descriptive Set Theory}, Springer, 1995.

\bibitem{KhN95}Khoussainov, B. and A. Nerode.  Automatic presentations of structures.  {\em LNCS}, {\bf 960}, 367-392, 1995.

\bibitem{knOQ} Khoussainov, B. and A. Nerode. Open Questions in the theory of automatic structures. {\em Bull. European Assoc. for Theoret. Comp. Sci.}, to appear, 2008.



\bibitem{KhRS03}Khoussainov, B., S. Rubin, and F. Stephan. On automatic partial orders.  {\em Proc. LICS '03}, 168-177, 2003.

\bibitem{KhRS04}Khoussainov, B., S. Rubin, and F. Stephan.  Definability and regularity in automatic structures. {\em Proc. STACS '04}, {\em LNCS} {\bf 2996}, 440-451, 2004.

\bibitem{KhNieRuSte04}Khoussainov, B., A. Nies, S. Rubin, and F. Stephan. Automatic Structures: Richness and Limitations.  {Proc. LICS '04}, 44-53, 2004. 


\bibitem{KLM}Khoussainov, B., J. Liu, and M. Minnes.  Unary Automatic Structures: An Algorithmic Perspective.  {\em Proc. TAMC '08}, LNCS {\bf 4978}, 548-559, 2008.

\bibitem{KM} Khoussainov, B. and M. Minnes.  Model Theoretic Complexity of Automatic Structures (Extended Abstract).  {\em Proc. TAMC '08}, LNCS {\bf 4978}, 2008.

\bibitem{Kleene}Kleene, S.C.  {\em Representation of events in nerve nets and finite automata}, Rand Corp., 1951.

\bibitem{KnM}Knight, J.F. and J. Millar.  Computable Structures of Rank $\omega_1^{CK}$.  Submitted to {\em J. Math. Logic}; posted on arXiv 25 Aug 2005.


\bibitem{KL06}Kuske, D. and M. Lohrey.  First-order and counting theories of omega-automatic structures. {\em Proc. FoSSaCS '06}, 322-336, 2006.

\bibitem{LiuPhD}Liu, J.  {\em Automatic Structures (provisional title)}, PhD Thesis, University of Auckland, in progress.


\bibitem{Mal61}Mal'cev, A.I.  Constructive Algebras. {\em Uspekhi Mathm. Nauk}, {\bf 16}, 3-60, 1961.

\bibitem{MinPhD}Minnes, M. {\em Automatic Structures (provisional title)}, PhD Thesis, Cornell University, 2008.

\bibitem{Nadel85} Nadel, M.E. $L_{\omega_{1},\omega}$ and admissible fragments.
{\em Model-Theoretic Logics} (K.J. Barwise and S. Feferman, eds.),  271-316,
1985.

\bibitem{NeRe90}Nerode, A. and J. Remmel. Polynomial time equivalence types, {\em Contemp. Math. (Proc. of Logic and computation '87)}, {\bf 106}, 221-249, 1990.

\bibitem{NiesSurvey} Nies, A. Describing Groups. {\em Bull. Symb. Logic}, {\bf 13}, 305-339, 2007.

\bibitem{Noskov} Noskov, G.A.  The elementary theory of a finitely generated almost solvable group. {\em Izv. Akad. Nauk SSSR Ser. Mat.} {\bf 47} , 498-517, 1983 (Russian); {\em Math. USSR Izv.} {\bf 22}, 465-482, 1984 (English translation).

\bibitem{OliverThomas}Oliver, G.P. and Thomas R.M. Automatic Presentations for Finitely Generated Groups.  {Proc. STACS '05}, {\em LNCS} {\bf 3404}, 693-704, 2005.

\bibitem{Presb29}Presburger, M. Uber die vollstandigkeit eines gewissen systems der arithmetic ganzer zahlen, {\em welchem die addition als einzige operation hervortiritt.  Compte-Rendus des Congres des Math. des pays Slavs}, 1929.

\bibitem{Rab60}Rabin, M.O.  Computable Algebra, General Theory and Theory of Computable Fields. {\em Trans. AMS}, {\bf 95}, 341-360, 1960.

\bibitem{Rab69}Rabin, M.O. Decidability of Second-Order Theories and Automata on Infinite Trees. {\em Trans. AMS}, {\bf 141}, 1-35, 1969.

\bibitem{Romanovskii} Romanovski$\breve{\i}$, N.S. On the elementary theory of an almost polycyclic group. {\em Math. Sb.}, {\bf 111}, 135-143, 1980 (Russian); {\em Math. USSR Sb.}, {\bf 39}, 1981 (English translation).

\bibitem{Rotman}Rotman, J. {\em An introduction to the theory of groups}.  Springer-Verlag, 1994.

\bibitem{RubPhD} Rubin, S.  {\em Automatic Structures}, PhD Thesis, University of Auckland, 2004.

\bibitem{RubinSurvey} Rubin, S. Automata presenting structures: A survey of the finite-string case.  {\em Bull. Symb. Logic}, to appear.

\bibitem{Shmelev} Shmelev G. Classification of indecomposable finite-dimensional representations of the Lie superalgebra $W(0, 2)$. {\em C. R. Acad. Bulgare Sci.}, {\bf 35:8}, 1025-1027, 1982 (Russian).

\bibitem{scanlon} Scanlon, T. Infinite Finitely Generated Fields are Biinterpretable with $N$.  {\em J. AMS}, 2008.

\bibitem{Sco65}Scott, D.  Logic with Denumerably Long Formulas and Finite Strings of Quantifiers.  {\em The Theory of Models} (J. Addison, L. Henkin, A. Tarski, eds.), 329-341, North-Holland, 1965.


\bibitem{Tar49}Tarski, A.  Arithmetical classes and types of Boolean algebras. {\em Bull. AMS}, {\bf 55}, 63, 1949.

\bibitem{Tar48} Tarski, A.  A decision method for elementary algebra and geometry.
{\em RAND Corporation}, Santa Monica, Calif.,  1948.

\bibitem{TW65}Thatcher, J.W. and J.B. Wright.  Generalized finite automata. {\em Notices AMS}, {\bf 12}, 820, 1965. 

\bibitem{Thomas}Thomas, W.  Automata on Infinite Objects. {\em Handbook of Theoretical Computer Science} (J. van Leeuwen, ed), 133-191, 1990.




\end{thebibliography}
\end{document}